\documentclass[11pt,reqno]{amsart}
\usepackage[top=1.2in, bottom=1.2in, left=1.3in, right=1.3in]{geometry}
\usepackage{amsfonts, amssymb, amsmath, enumerate, bm, xspace}
\usepackage[linktoc=page, colorlinks, linkcolor=blue, citecolor=blue,urlcolor = black]{hyperref}
\usepackage{pbox,url,array,makecell,algorithm2e}

\usepackage{epsfig,subcaption,graphicx}
\graphicspath{{figures/}}
\usepackage{epstopdf}
\usepackage{caption}
\usepackage{floatrow,pdfsync} %

\numberwithin{equation}{section}

\DeclareMathOperator{\E}{\mathbb{E}}

\renewcommand{\Pr}[2][]{\mathbb{P}_{#1} \left\{ #2 \rule{0mm}{3mm}\right\}}

\def \P {\mathbb{P}}
\def \R {\mathbb{R}}

\def \a {\alpha}
\def \b {\beta}

\def \e {\varepsilon}
\def \d {\delta}

\def \tran {\mathsf{T}}

\def \onevector {{\bf 1}}

\newtheorem{theorem}{Theorem}[section]

\newtheorem{corollary}[theorem]{Corollary}
\newtheorem{lemma}[theorem]{Lemma}

\theoremstyle{remark}

\newtheorem{algo}{Algorithm}

\begin{document}

\title[]{Edge sampling using network local information}

\author{Can M. Le}
\address{University of California, Davis}
\email{canle@ucdavis.edu}

\thanks{}


\begin{abstract}
Edge sampling is an important topic in network analysis. It provides a natural way to reduce network size while retaining desired features of the original network. Sampling methods that only use local information are common in practice as they do not require access to the entire network and can be parallelized easily.        Despite promising empirical performances, most of these methods are derived from heuristic considerations and therefore still lack theoretical justification. To address this issue, we study in this paper a simple edge sampling scheme that uses network local information. We show that when local connectivity is sufficiently strong, the sampled network  satisfies a strong spectral property.  
We quantify the strength of local connectivity by a global parameter and relate it to more common network statistics such as the clustering coefficient and network curvature. 
Based on this result, we also provide sufficient conditions under which random networks and hypergraphs can be sampled efficiently. 
\end{abstract}


\maketitle

\section{Introduction}

Network analysis has become an important area in many research domains. It provides a natural way to model and analyze data with a complex interdependence among entities. A network typically consists of a set of nodes representing the entities of interest and a set of edges between nodes encoding the relations between the nodes. 
For example, in a social network such as Facebook or Twitter, nodes are users and there is an edge between two users if they are friends.
Studying the structure of a network provides valuable information about how entities interact and may help predict the formation of different groups \cite{Goldenberg2010,Fortunato2010}.

As real-world networks are often very large, it is difficult and often impossible to store or even get access to the entire data set. It is therefore desirable to preprocess the data to reduce the network size before performing any analysis. A natural method for this task is graph sparsification, a well-known edge sampling method in network literature  \cite{Benczur&Karger1996,Spielman&Teng2004,Spielman&Shrivastava2011}.
For a network of $n$ nodes, one samples edges independently with probabilities proportional to their effective resistances, i.e. the electrical resistances between the same nodes in the resistor network obtained from the original network by replacing edges with resistors  of unit conductance \cite{Ghosh&Boyd&Saberi2008}. It has been shown that sampling and storing $O(n\log n)$ weighted edges is sufficient for approximately preserving the important topological structure of the original network  \cite{Spielman&Shrivastava2011}. Specifically, for an undirected network $G=(V,E)$ with the set of nodes $V=\{1,2,...,n\}$ and the set of edges $E\subseteq V\times V$, let $A$ be the adjacency matrix with $A_{ij}=1$ if $(i,j)\in E$ and $A_{ij}=0$ otherwise. Let $L_G = D - A$ be the Laplacian, where $D$ is the diagonal matrix with node degrees $d_i = \sum_{j\in V} A_{ij}$ on the diagonal, and define the Laplacian $L_H$ for the weighted network $H$ output by graph sparsification in a similar way. Then $H$ satisfies the following inequality for every $x\in\R^n$, known as the strong spectral property:  
\begin{equation}\label{eq: spectral property}
(1-\e) x^\tran L_G x \le x^\tran L_H x \le (1+\e) x^\tran L_G x.
\end{equation}  
Although this method has a strong theoretical guarantee, a serious drawback it suffers from, especially when applied to very large networks, is that it requires access to the entire network for computing effective resistances of all edges. Also, the computation involves a complicated linear system solver of Spielman and Teng, not easy to implement in practice.
Although some improvements of 
\cite{Spielman&Shrivastava2011} have been proposed, they still rely on complicated linear system solvers \cite{Kelner&Levin2013,Kapralovetal2014}. 

To avoid this problem, several fast and simple edge sampling methods have been developed with more emphasis on preserving certain network features such as the number of connected components, network diameter, homophily, node centrality measures or community structure \cite{Newman2010}. One of the simplest sampling methods is uniform sampling, which samples edges independently and uniformly at random \cite{Sadhanala&Wang&Tibshirani2016,Li&Levina&Zhu2016}.
More adaptive methods leverage the strong local connectivity of networks that is widely observed in practice: the network neighborhoods of most of the nodes are surprisingly dense \cite{Watts&Strogatz1998,Uganderetal2011}.  They sample edges according to certain edge scores that can be calculated locally without access to the entire network such as the Jaccard similarity score \cite{Satuluri&Parthasarathy&Ruan2011}, the number of triangles \cite{Hamann.et.al.2016} or the number of quadrangles containing the edges under consideration \cite{Nocaj&Ortmann&Brandes2014}; see also \cite{Hamann.et.al.2016} for methods based on other local measures. Although these methods have been empirically shown to perform well and can be parallelized easily, to our best knowledge, there is still no theoretical guarantee for their performances. It is also unclear if other features of networks (besides the targeting features  considered) are preserved.  

In an attempt to understand the theoretical properties of these methods, in this paper we study  a fast and simple edge sampling scheme similar to methods that use Jaccard similarity or number of triangles \cite{Satuluri&Parthasarathy&Ruan2011,Hamann.et.al.2016}. Specifically, for an undirected network $G=(V,E)$,
we sample each edge $(i,j)\in E$ with probability inversely proportional to the number of common neighbors of $i$ and $j$ (i.e. those nodes connected to both $i$ and $j$). The numbers of common neighbors have been used in network literature, for example in the context of community detection \cite{Rohe&Qin2013} and network embedding \cite{Papadopoulos&Aldecoa&Krioukov2015}.

We observe that when the numbers of common neighbors are sufficiently large compared to node degrees, our sampling method satisfies the same strong spectral property \eqref{eq: spectral property} that the graph sparsification does, while avoiding the complicated calculation of effective resistances. This result also provides theoretical evidence supporting edge sampling methods based on local statistics \cite{Satuluri&Parthasarathy&Ruan2011,Nocaj&Ortmann&Brandes2014,Hamann.et.al.2016}. Qualitatively, as the number of common neighbors increases, the network local connectivity gets stronger and our sampling method becomes more similar to graph sparsification using effective resistances. In contrast, as the numbers of common neighbors decrease, the method becomes more similar to uniform sampling.
We quantify the strength of the network local connectivity by the following parameter 
\begin{equation}\label{eq: main assumption}
\a = \frac{1}{n}\sum_{(i,j)\in E} \frac{2}{t_{ij}+2},
\end{equation}
where $t_{ij}$ denotes the number of common neighbors of node $i$ and node $j$. As we will show, $\alpha$ is closely related to other well-known and similar in nature statistics such as the
clustering coefficient \cite{Watts&Strogatz1998} and network curvature \cite{Bauer&Urgen&Liu2012}; see Section~\ref{sec: bound on alpha} for the definition. More importantly, it determines the sample size (the number of sampled edges) needed for the strong spectral property to hold. 

\subsection{Our contributions}
We make the following contributions in this paper. First, in Section~\ref{sec: theory} we propose  a simple sampling method by leveraging the strong local connectivity that has been often observed for real-world networks and show that it satisfies the strong spectral property \eqref{eq: spectral property} if we sample $O(\alpha n\log n)$ edges, where $\alpha$ is defined by \eqref{eq: main assumption}. As a direct consequence, we show that uniform sampling with replacement also satisfies \eqref{eq: spectral property} if the sample size is sufficiently large; the exact value is given by \eqref{eq: uniform sample size}. This  requirement can be relaxed if 
a hybrid sampling method that combines both uniform sampling and sampling according to the number of common neighbors is used. 
Second, we provide lower and upper bounds on $\alpha$ for general networks in terms of the clustering coefficient and network curvature (Section~\ref{sec: bound on alpha}). Since $\alpha$ directly determines the sample size required for the strong spectral property, these bounds provide useful information about when our sampling method can be used efficiently. They also show a connection with other sampling methods that use different local statistics \cite{Satuluri&Parthasarathy&Ruan2011,Nocaj&Ortmann&Brandes2014,Hamann.et.al.2016} for which the theory developed in this paper may potentially  be applied. 
Third, in Section~\ref{sec: random graphs} we provide  an upper bound on $\alpha$ for the general inhomogeneous Erd\H{o}s-R\'{e}nyi random graph model \cite{Bollobas2007}. Since this model is very popular in network literature, the bound provides a rich class of examples for which our sampling method can be used for reducing the network size. We discuss in Section~\ref{sec: hypergraphs} another natural class of examples, the hypergraphs, for which our method can be found useful.      
Lastly, in Section~\ref{eq: simulation} we show  that $\alpha$ is small for many real-world networks and perform a thorough numerical study to evaluate our sampling method.   

\subsection{Related work}
The simplest sampling method is bond percolation, which independently selects edges  
with a fixed probability $\varepsilon$ \cite{Alon&Benjamini&Stacey2004,Nachmias2009,
Bollobas&Borgs&Chayes&Riordan2010}. 
If $\varepsilon$ is sufficiently large so that $\Omega(n\log n)$ edges are selected then with high probability the adjacency matrix of the sparsified network concentrates around $\varepsilon A$ by a standard matrix concentration result \cite{Oliveira2010}. 
The advantage of this method is that it is fast and only requires the total number of edges in the network as a global input parameter. However, it satisfies a much weaker property than the strong spectral property \cite{Spielman&Teng2004}. A closely related method is uniform sampling, for which \cite{Sadhanala&Wang&Tibshirani2016} shows that \eqref{eq: spectral property} holds with high probability, but only for smooth vectors $x$.  

In semi-streaming setting, \cite{Benczur&Karger1996} and \cite{Goel&Kapralov&Khanna2010} show that local network structure can be used to design sampling methods that approximately preserve all cuts of the original network; here, the cut of a set of nodes is the number of edges between that set and its complement in $V$. However, this property is strictly weaker than the strong spectral property that our method satisfies \cite{Kelner&Levin2013}.

\section{Edge sampling using common neighbors}\label{sec: theory}


For an undirected network $G=(V,E)$ and $(i,j)\in E$, let $t_{ij}$ be the number of common neighbors of $i$ and $j$. For simplicity of presentation, we first discuss the case when $t_{ij}$ are known for all edges. In practice, they can be either  exactly calculated in a parallel manner or approximated by neighbor sampling; see Section~\ref{sec: Tij calculation}  for a more detailed discussion. 
To form a sparsifier $H$, we sample $m$ edges of $G$ independently according to a multinomial distribution with probabilities
\begin{equation}\label{eq: pij}
p_{ij} = \frac{\frac{2}{t_{ij}+2}}{\sum_{(i,j)\in E}\frac{2}{t_{ij}+2}}.
\end{equation}
If an edge $(i,j)\in E$ is selected $k\ge 1$ times then we add it to $H$ and assign the weight $k(mp_{ij})^{-1}$ to it.   

Note that $2/(t_{ij}+2)$ is the effective resistance of the edge between $i$ and $j$ in a subgraph of $G$ consisting of the  edge $(i,j)$ and $t_{ij}$ paths of length two between $i$ and $j$. It is therefore an upper bound of the effective resistance of the edge between $i$ and $j$ in $G$; for a detailed explanation, see the proof of Theorem~\ref{thm: uniform edge sampling} in Appendix~\ref{subsec: proof of main theorems}. 
 
The following theorem shows that our sampling method satisfies the strong spectral property.   

\begin{theorem}[Sampling method using $t_{ij}$]\label{thm: main theorem}
Consider an undirected and connected network $G=(V,E)$. Let $\e\in(0,1)$ and $\alpha$ be the parameter of $G$ defined by \eqref{eq: main assumption}.
Form a weighted network $H$ by sampling $8\a n\log n/\e^2$ edges of $G$ as described above. Then $H$ satisfies the strong spectral property \eqref{eq: spectral property} with probability at least $1-1/n$. 
\end{theorem}

Parameter $\a$ measures the average strength of network local connectivity. To better understand $\a$, consider a special case when 
$d_i=d$ for all vertices $i$ and $t_{ij}=t$ for all edges $(i,j)\in E$. Then $\a\approx 2|E|/(nt)=d/t$, where here and after we use $|\mathcal{M}|$ to denote the number of elements of the set $\mathcal{M}$.  Thus, if $(i,j)\in E$ then the number of common neighbors of $i$ and $j$ is approximately $d/\alpha$. In other words, $i$ and $j$ share a fraction of $1/\a$   of their neighbors.

When the local connectivity is strong, i.e. $\alpha = O(1)$, Theorem~\ref{thm: main theorem} shows that we can approximately preserve the network topology if we locally sample and retain $O(n\log n)$ edges. In contrast, if the local connectivity is weak (for example when $t_{ij}=O(1)$) then $p_{ij}$ are of the same order, resulting in a  sampling scheme similar to uniform sampling. Table~\ref{tb: stats} shows the value of $\a$ and the clustering coefficient (see Section~\ref{sec:clustering} for the definition) for several well-known real-world networks. Note that while these networks are relatively sparse, the values of $\a$ are quite small, which suggests that real-world networks have strong local connectivity. 

The above sampling method requires access to the number of common neighbors $t_{ij}$ for all pairs of incident nodes. If $t_{ij}$ are readily available, which is the case for some social networks such as Facebook, then the computational complexity of this sampling method is linear in the total number of edges $|E|$. When $t_{ij}$ are not available, we can calculate them in parallel fashion or estimate them by neighbor sampling; see Section~\ref{sec: Tij calculation}   for more detail. The following theorem shows that the strong spectral property still holds if we use estimates of  $t_{ij}$ and increase the sample size by a factor depending on the accuracy of the estimates.  

\begin{theorem}[Sampling method using estimates of $t_{ij}$]\label{thm: uniform edge sampling}
Consider an undirected and connected network $G=(V,E)$ and let $\hat{t}_{ij}$ be nonnegative estimates of $t_{ij}$ such that
\begin{equation}\label{eq: estimation accuracy}
\hat{t}_{ij}+2 \le C(t_{ij}+2)
\end{equation}
for all edges $(i,j)\in E$ and some constant $C$. Let $\e\in(0,1)$ and denote
\begin{equation}\label{eq: alpha hat}
\hat{\a} = \frac{1}{n}\sum_{(i,j)\in E_G} \frac{2}{\hat{t}_{ij}+2}.
\end{equation}
Form a weighted graph $H$ by sampling $8C\hat{\a} n\log n/\e^2$ edges of $G$ as described in Theorem~\ref{thm: main theorem} but using $\hat{t}_{ij}$ instead of $t_{ij}$. Then $H$ satisfies the spectral property \eqref{eq: spectral property} with probability at least $1-1/n$.
\end{theorem}

%

The proof of Theorem~\ref{thm: uniform edge sampling} depends crucially on condition \eqref{eq: estimation accuracy}. It implies that $2/(t_{ij}+2)\le 2C/(\hat{t}_{ij}+2)$ and consequently the effective resistance of the edge $(i,j)$ is bounded by $2C/(\hat{t}_{ij}+2)$. This observation allows us to express the Laplacian of the sparsified network as a sum of independent matrices with spectral norms bounded by $C\hat{\alpha}n$ up to a scaling matrix factor. A standard matrix concentration result is then used to show the strong spectral property; see the proof in Appendix~\ref{subsec: proof of main theorems} for more detail. Note that  
Theorem~\ref{thm: main theorem} follows directly from 
Theorem~\ref{thm: uniform edge sampling} by setting $\hat{t}_{ij} = t_{ij}$ and $C=1$.

One may wonder how many edges must be sampled so that the uniform sampling (which samples edges with probabilities $p_{ij}=1/|E|$) satisfies the spectral property \eqref{eq: spectral property}. The uniform sampling is obtained by setting $\hat{t}_{ij} = t$ for all edges of $G$ in Theorem~\ref{thm: uniform edge sampling}. The constant $C$ can be taken to be
$$
C = \frac{t+2}{\min_{(i,j)\in E} t_{ij}+2} \quad \text{and} \quad \hat{\a} = \frac{2|E|}{n(T+2)}.
$$
Therefore by Theorem~\ref{thm: uniform edge sampling}, the uniform sampling satisfies \eqref{eq: spectral property} with high probability if the sample size is
\begin{equation}\label{eq: uniform sample size}
m = \frac{16\varepsilon^{-2}|E|\log n}{\min_{(i,j)\in E} t_{ij}+2}.
\end{equation}
If $\min_{(i,j)\in E} t_{ij}$ is of order $|E|/n$, i.e. numbers of common neighbors are at least a constant fraction of the average degree, then $m=O(n\log n)$. 

In general, the sample size requirement \eqref{eq: uniform sample size} is optimal up to the logarithm and constant factors. That is, \eqref{eq: spectral property} needs not hold if $m=o(|E|/(\min_{(i,j)\in E} t_{ij}+2))$. To see this, consider an example of a graph $G$ consisting of a complete graph of $n-1$ nodes and a node $i$ of degree $k=o(n)$. Then $\min_{(i,j)\in E} t_{ij} = k-1$. If $m=o(|E|/(k+1))$ and edges of $G$ are sampled uniformly then the probability that no edges incident to $i$ is selected is
$$
\left(1-\frac{k}{|E|}\right)^m = \left(1-\frac{k}{|E|}\right)^{\frac{|E|}{k}\cdot \frac{mk}{|E|}} \approx \exp\left(-\frac{mk}{|E|}\right) \approx 1.
$$  
That is, with probability close to one, $i$ is an isolated node in the (weighted) sparsified graph. Therefore the degree of $i$ cannot be approximately preserved, which implies that the spectral property \eqref{eq: spectral property} does not hold.

For graphs with small value of $\min_{(i,j)\in E} t_{ij}$, the sample size $m$ in \eqref{eq: uniform sample size} for the uniform sampling scheme may get as large as the total number of edges $|E|$, which defies the purpose of graph sparsification.
By choosing $\hat{t}_{ij} = t$ only if $t_{ij}>t$ for some large threshold $t$ and $\hat{t}_{ij} = t_{ij}$ if $t_{ij}\le t$, we obtain a hybrid of uniform sampling and sampling using common neighbors that may require smaller sample size than   \eqref{eq: uniform sample size}. Indeed, in  Theorem~\ref{thm: uniform edge sampling}, we can choose $C=1$ and
\begin{eqnarray*}
\hat{\alpha} = \frac{1}{n} \sum_{(i,j)\in E: t_{ij}\le t} \frac{2}{t_{ij}+2} + \frac{1}{n} \sum_{(i,j)\in E: t_{ij}> t} \frac{2}{t+2}
\le \alpha + \frac{2|E|}{n(t+2)}.
\end{eqnarray*}
The required sample size for the hybrid method to obtain the spectral property \eqref{eq: spectral property} with high probability is then
\begin{eqnarray*}
8\varepsilon^{-2} C\hat{\a} n\log n \le 8\varepsilon^{-2}\left(\alpha+\frac{2|E|}{n(t+2)}\right)n\log n 
= 8\varepsilon^{-2} \alpha n\log n + \frac{16|E|\log n}{\varepsilon^2(t+2)},
\end{eqnarray*} 
which is smaller than the sample size in \eqref{eq: uniform sample size} if $\alpha n = o(|E|/(\min_{(i,j)\in E} t_{ij}+2))$ and $\min_{(i,j)\in E} t_{ij} = o(t)$. This hybrid method illustrates an interesting application of Theorem~\ref{thm: uniform edge sampling} and may also be useful when uniform sampling is desirable, for example for controlling the variance of the sparsified graph.

\section{Bounding parameter $\a$}\label{sec: bound on alpha}
In this section we draw the connection between the parameter $\alpha$ and two of the most common network statistics, the clustering coefficient \cite{Watts&Strogatz1998} and the network curvature \cite{Bauer&Urgen&Liu2012}. 

\subsection{Lower bound}\label{sec:clustering}

It has been observed that for many real-world networks, the neighborhoods of most of the nodes are surprisingly dense \cite{Watts&Strogatz1998,Uganderetal2011}. This reflects the belief that incident nodes exhibit the transitivity property: if $i$ and $j$ are connected and $j$ and $k$ are connected then it is likely that $i$ and $k$ are also connected.  
One way to measure the transitivity is via the clustering coefficient \cite{Watts&Strogatz1998}. 
For an undirected network $G = (V,E)$, the local clustering coefficient of node $i\in V$ is defined as the ratio between the number of triangles containing $i$ and the maximum number of triangles it can form with incident nodes
$$
c_i = \frac{|\{ (j,k)\in E: (i,j)\in E, (i,k)\in E \}|}{d_i(d_i-1)/2}.
$$   
The clustering coefficient of a network $G$ is the average of all local clustering coefficients
$$
c = \frac{1}{n} \sum_{i=1}^n c_i.
$$

The following theorem provides a lower bound on parameter $\a$ in terms of the clustering coefficient $c$ and node degrees $d_i$. It shows that if node degrees are large and $c$ is small then
$\alpha$ is large and therefore a large sample size is required for our method to obtain the spectral property \eqref{eq: spectral property}. On the other hand, Table~\ref{tb: stats} suggests that $\alpha$ is small when $c$ is large. Since the clustering coefficient is a very popular statistic and has been calculated for most of available real-world networks, the connection to the clustering coefficient provides valuable information about $\alpha$ before the sampling procedure is performed.

\begin{theorem}[Lower bound on $\a$]\label{lem: lower bound resistance sum}
For any undirected and connected network we have
\begin{equation}\label{eq: lower bound rs}
\a \ge\frac{1}{4c+ \frac{2}{n}\sum_{i=1}^n\frac{1}{d_i}}.
\end{equation}
\end{theorem}

According to Theorem~\ref{lem: lower bound resistance sum}, if $c\gtrsim 1/n\sum_{i\in V} 1/d_i$ then $\a$ satisfies $\a \gtrsim 1/c$ (for two sequences $a_n$ and $b_n$, we write $a_n\gtrsim b_n$ if $a_n\ge C b_n$ for some constant $C$ and sufficiently large $n$). The geometric random graph model described in Corollary~\ref{cor: geometric graph} below provides examples for which the upper bound $\a \lesssim 1/c$ also holds; for more detail, see the discussion following Corollary~\ref{cor: geometric graph}. In addition, Table~\ref{tb: stats} gives examples of real networks for which $\alpha$ and $1/c$ are of similar order. 

There exist graphs for which the two sides of \eqref{eq: lower bound rs} 
are of different orders. For example, let $G = K_n\cup E_n$ be the union of a complete graph $K_n$ of size $n$ and an Erd\H{o}s-R\'{e}nyi random graph $E_n$, also of size $n$, for which edges are formed independently between each pair of nodes with probability $d/n$; we connect $K_n$ and $E_n$ by an arbitrary edge to make $G$ a connected graph. 
If $\sqrt{n}\lesssim d = o(n)$ then 
an easy calculation shows that with high probability, the left hand-side of \eqref{eq: lower bound rs} is of order $n/d$ while the right hand-side is bounded.

\subsection{Upper bound}\label{sec: Ricci curvature}

Another measure of network transitivity that has recently attracted much attention is the network curvature \cite{Bauer&Urgen&Liu2012,Jost&Liu2014,Lin&Lu&Yau2014,Bhattacharya&Mukherjee2015}. In this section we recall the definition of network curvature and show that if it is bounded from below by some constant $\kappa_0>0$ then $\a\le 1/\kappa_0$.

Denote by $d(i,j)$ the length of a shortest path connecting nodes $i$ and $j$. For each node $i$, consider a uniform measure $m_i$ with support being the set $N_i$ of neighbors of $i$:
$$
m_i(k) = 
\left\{
  \begin{array}{ll}
    \frac{1}{d_i}, & \hbox{if } k\in N_i\\
    0, & \hbox{otherwise.}
  \end{array}
\right.
$$
The optimal transportation distance between $m_i$ and $m_j$ is defined as follows:
$$
W_1(m_i,m_j) = \inf_{\xi\in\Pi (m_i,m_j)}\sum_{(k,k')\in V\times V} d(k,k')\xi(k,k'),
$$ 
where $\Pi(m_i,m_j)$ is the set of all probability measures on $V\times V$ with marginals $m_i$ and $m_j$. Intuitively, $\xi(k,k')$ represents the mass transported from $k$ to $k'$, and $W_1(m_i,m_j)$ is the optimal cost for moving a unit mass distributed evenly among neighbors of $i$ to neighbors of $j$. 
With this notion of distance between probability measures on $G$, the curvature $\kappa$ defined for every pair of nodes $i$ and $j$ is
$$\kappa(i,j) = 1 - \frac{W_1(m_i,m_j)}{d(i,j)}.$$ 

To illustrate, in Figure~\ref{fig: karate} we show  the Zachary's karate club network \cite{Zachary1977} together with the information of its curvatures for incident nodes. In particular, edges with negative curvatures are in blue, positive curvatures -- in red and zero curvatures -- in black; widths of edges are proportional to magnitudes of curvatures. 
\begin{figure}[!ht]
  \centering
  \includegraphics[trim=80 30 50 40,clip,width=0.4\textwidth]{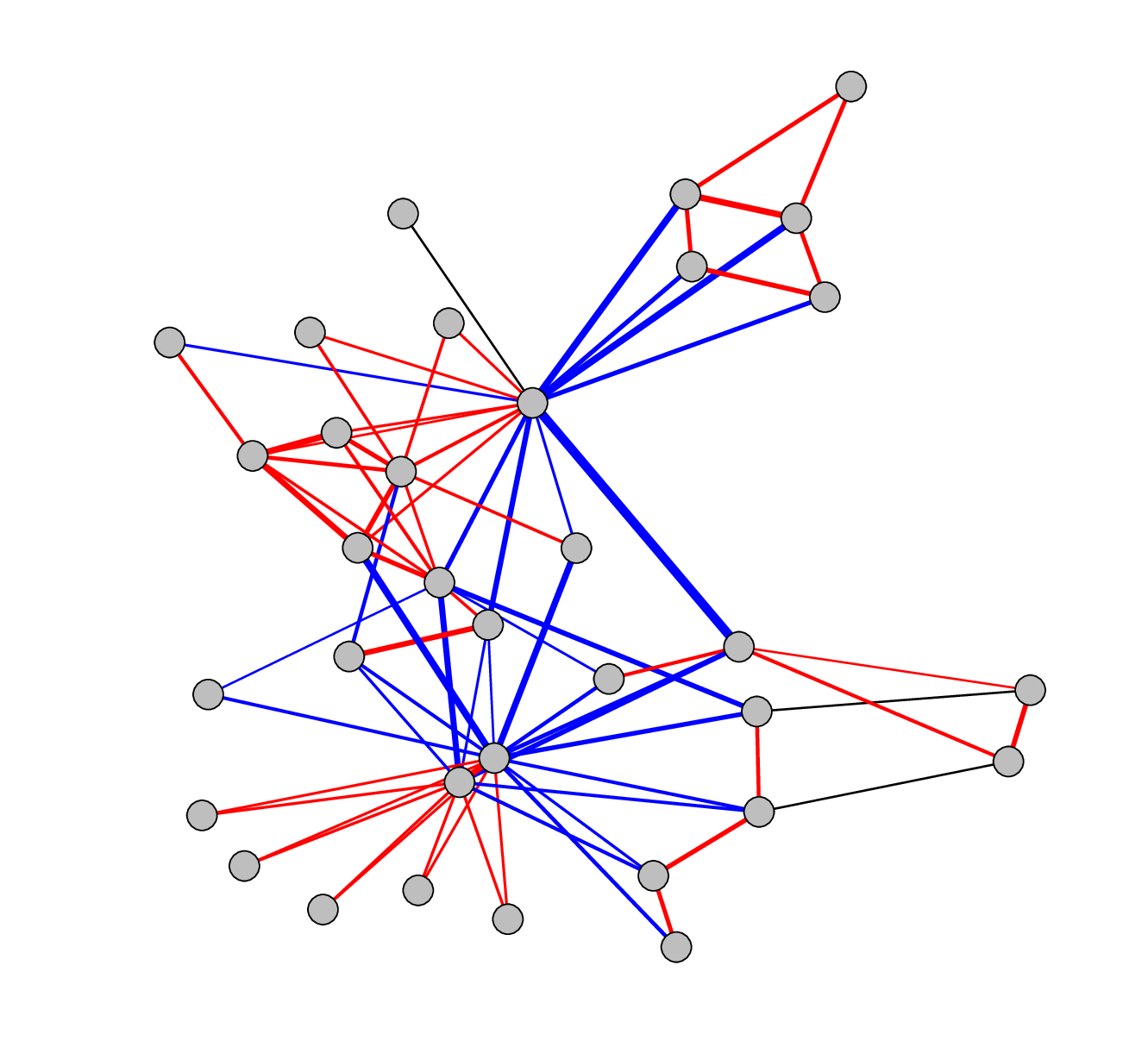}\\ 
  \caption{Zachary's karate club network \cite{Zachary1977}. Edges with negative curvatures are in blue, positive curvatures -- in red and zero curvatures -- in black; widths of the edges are proportional to the magnitudes of their curvatures.}
  \label{fig: karate}
\end{figure}

We say $\kappa\ge \kappa_0$ for some constant $\kappa_0$ if $\kappa(i,j)\ge \kappa_0$ for every pair of nodes $i$ and $j$. If $\kappa\ge \kappa_0$ then by definition $W_1(m_i,m_j) \le (1-\kappa_0) d(i,j)$ for all $(i,j)\in V\times V$. 
In particular, if $i$ and $j$ are connected then $W_1(m_i,m_j) \le 1-\kappa_0$. Note that if $G$ is a connected graph then the inverse is also true: If $W_1(m_i,m_j) \le 1-\kappa_0$ holds for all pairs of connected  nodes $i$ and $j$ then  $W_1(m_i,m_j) \le (1-\kappa_0) d(i,j)$ holds for all $(i,j)\in V\times V$ by a triangle inequality. 


This notion of curvature is closely related to the simple random walk on a network. If $\kappa\ge \kappa_0>0$ then \cite{Ollivier2009} shows that the spectral gap between the two largest eigenvalues of the transition matrix $D^{-1}A$ is bounded from below by $\kappa_0$ (see also \cite{Bauer&Urgen&Liu2012} for an improvement of the bound). Thus, the curvature of a graph controls how fast a simple random walk on that network mixes.  

The following theorem provides a simple upper bound on $\alpha$ in terms of the curvature.

\begin{theorem}[Upper bound on $\a$]\label{lem: upper bound er}
Let $G$ be an undirected and connected network. Assume there exist constants $\kappa_0>0$ and $C>0$ such that
$\kappa(i,j)\ge \kappa_0$ for all but at most $Cn$ edges of $G$. Then $\alpha \le 1/\kappa_0 + C$. 
\end{theorem}

\section{Random networks}\label{sec: random graphs}

In this section we provide a high probability bound on $\alpha$ for inhomogeneous Erd\H{o}s-R\'{e}nyi random networks \cite{Bollobas2007} satisfying some mild conditions. As a corollary, we give an example of a geometric random network model for which $\alpha$ is bounded. 

\begin{theorem}[Inhomogeneous Erd\H{o}s-R\'{e}nyi networks]\label{thm: alpha bound random graph}
Consider a random graph with adjacency matrix $A$ such that the upper diagonal elements of $A$ are independent Bernoulli random variables. Denote $P=\E A$ and $\Delta = \max_{i}\sum_{j=1}^nP_{ij}$. Assume that there exists a sufficiently large constant $C$ such that
\begin{equation}\label{eq: max condition}
\Delta\ge C\log n \qquad \text{and} \qquad \Delta\cdot \left[1+\max_{i,j}\left(P^2\right)_{ij}\right] \le \frac{1}{C\log n} \sum_{i<j}P_{ij}.
\end{equation}
Then with probability at least $1-1/n$, \begin{equation}\label{eq: alpha bound random graph}
\alpha \le \frac{1}{n} \sum_{i<j} \frac{10 P_{ij}}{\E t_{ij}+2}.
\end{equation}
In particular, if the right-hand side of \eqref{eq: alpha bound random graph} is bounded then $\alpha$ is also bounded. 
\end{theorem}

The first inequality of \eqref{eq: max condition} requires that the maximal expected node degree grow at least as $\log n$; this is a natural condition because otherwise the network would already be sparse and no sampling would be  needed. The second inequality of \eqref{eq: max condition} is a condition on the maximal expected degree $\Delta$, the maximal expected number of common neighbors $\max_{ij}(P^2)_{ij}$ and the expected number of edges $1/2\sum_{i<j}P_{ij}$. If all nodes in the graph are of similar expected degree then $\sum_{i<j}P_{ij}\approx n\Delta$. Therefore, using the crude bound $(P^2)_{ij} \le \Delta$, the second inequality of \eqref{eq: max condition} is satisfied if $\Delta$ is at most of order $n/\log n$. 

By Jensen's inequality and the independence between $A_{ij}$ and $t_{ij}$, we have
$$
\E \a = \frac{1}{n}\sum_{i<j} \E \frac{2A_{ij}}{t_{ij}+2} = \frac{1}{n}\sum_{i<j} \E \frac{2P_{ij}}{t_{ij}+2} \ge \frac{1}{n}\sum_{i<j} \frac{2P_{ij}}{\E t_{ij}+2}.
$$    
It then follows from \eqref{eq: alpha bound random graph} that $\alpha\le 5\E \alpha$ with probability at least $1-1/n$, while naively applying Markov's inequality gives the same inequality with probability at least $4/5$. 
Note, however, that the upper bound of \eqref{eq: alpha bound random graph} is much easier to calculate than $\E \alpha$. 

As a direct consequence of Theorem~\ref{thm: alpha bound random graph}, the following corollary shows that $\alpha$ is bounded for a simple geometric random network model.

\begin{corollary}[Geometric random networks]\label{cor: geometric graph}
Let $X=\{x_1,x_2,...,x_n\}\subseteq K$ be a set of points in a bounded set $K\subseteq \R^d$ with unit volume. For each pair of nodes $(i,j)$, let
$$
P_{ij} = 
\begin{cases}
\delta, & \text{if } \|x_i-x_j\|\le r_n,\\
0, & \text{otherwise}.
\end{cases}
$$
Denote by $n_i$ the number of points of $X$ of distance at most $R_n$ from $x_i$; similarly, denote by $n_{ij}$ the number of points of $X$ of distance at most $R_n$ from $x_i$ and $x_j$. Assume that there exists a constant $C = C(d,K)$ depending only on $d$ and $K$ such that for every node $i$ and every node $j$ with $\|x_j-x_i\|\le r_n$,
\begin{eqnarray}\label{eq: geometric graph condition}
C^{-1}
n r_n^d \le n_i,n_{ij}\le C n r_n^d \quad \text{and} \quad C^2\delta^{-1}\log n \le nr_n^d \le \frac{n}{4C^3\delta^2\log n}.  
\end{eqnarray}
Then $\alpha \le 5C^2/\delta$ with probability at least $1-1/n$. 
\end{corollary}

The first condition of \eqref{eq: geometric graph condition} holds with high probability if $x_1,...,x_n$ are independently drawn from a uniform distribution on $K$. Indeed, for every node $i$, $n_i/n$ is approximately the volume of the ball of radius $r_n$ and center $x_i$, which is proportional to $r_n^d$ up to a constant depending on $d$ and $K$; a similar argument holds for $n_{ij}$ with $\|x_j-x_i\|\le r_n$. The second condition of \eqref{eq: geometric graph condition} requires that the average degree of the graph be roughly between $\log n$ and $n/\log n$. If these conditions are  satisfied then $\alpha$ is bounded with high probability.

Theorem~\ref{lem: lower bound resistance sum} shows that $\a \gtrsim 1/c$ if the clustering coefficient $c$ is at least of the same order as $1/n\sum_{i\in V} 1/d_i$. Corollary~\ref{cor: geometric graph} provides examples for which the reverse bound  also holds. Indeed, since $\alpha \lesssim 1/\delta$ with high probability by Corollary~\ref{cor: geometric graph}, the bound $\alpha\lesssim 1/c$ holds if $c\gtrsim \delta$ with high probability. To see why that is the case, for every node $i$ let $N_i$ be the set of all neighbors of $i$. Then conditioned on $N_i$, the probability that two neighbors of $i$ are connected is at most $\delta$. Therefore the number of triangles containing $i$ is stochastically bounded by a sum of $d_i(d_i-1)/2$ independent Bernoulli random variables with success probability $\delta$. Using a standard concentration result and union bounds, we see that the local clustering coefficients satisfy $c_i \lesssim \delta$ for all $i$ with high probability. Since $c$ is the average of $c_i$, this implies $c \gtrsim \delta$ with high probability.

\section{Sampling hypergraphs}\label{sec: hypergraphs}


Strong local connectivity of a network is often caused by the fact that each node belongs to one or several tightly connected small groups \cite{Gupta&Roughgarden&Seshadhri2014}.
To simplify the analysis, we assume that within each small group, all nodes are connected. Under this assumption, a network can be modeled by a hypergraph $\mathcal{G}=(V,\mathcal{E})$ which consists of a set of nodes $V$ and a set of hyperedges $\mathcal{E}$ where each hyperedge is a subset of $V$. In this section we derive  a condition under which a hypergraph can be sampled and reduced to a weighted network. 
This provides another example for which our sampling scheme works well and may be useful in practice as a computational acceleration technique. 

The Laplacian previously defined for networks can be naturally extended to hypergraphs through clique expansion \cite{Rodriguez2002,Agarwal&Branson&Belongie2006}.
For a hypergraph $\mathcal{G}=(V,\mathcal{E})$, the evaluation of the Laplacian $L_\mathcal{G}$ at a vector $x$ is defined by
$$
L_{\mathcal{G}}(x) = \sum_{e\in\mathcal{E}}\sum_{i,j\in e}  (x_i-x_j)^2.
$$
If we view $x$ as a function from $V$ to $\R$ then $L_{\mathcal{G}}(x)$ measures the smoothness of $x$ and it occurs naturally in many problems of estimating smooth functions 
\cite{Smola&Kondor2003,Belkin&Matveeva&Niyogi2004,Huang&Ma&Li&Zhang2011,Kirichenko&vanZanten2017,Li&Levina&Zhu2016,
Le&Li.netreg2020}.

Let $G=(V,E,W)$ be a weighted network such that $(i,j)\in E$ if and only if both $i$ and $j$ belong to at least one hyperedge of $\mathcal{G}$, and $W$ denotes the weight matrix with entries $W_{ij}$ being the number of hyperedges that both $i$ and $j$ belong to. It is easy to see that $L_{\mathcal{G}}(x) = x^\tran L_{G}x$ for every $x$, where $L_G$ is the Laplacian of the weighted network $G$ defined by
$$x^\tran L_G x = \sum_{(i,j)\in E} W_{ij}(x_i-x_j)^2.$$ 
Thus, if we are mainly interested in the smoothness of functions determined by $\mathcal{G}$ then we can replace $\mathcal{G}$ with $G$. We  call $G$ the weighted network induced by $\mathcal{G}$.

To form a sparsifier $H = (V,E_H,W_H)$ of $G$, we sample with replacement $m$ edges of $G$ with probability 
\begin{equation*}\label{eq: pij weighted graph}
\mathcal{P}_{ij} = \frac{\tilde{t}_{ij}^{-1}}{\sum_{(i,j)\in E} \tilde{t}_{ij}^{-1}}, \quad \text{where} \quad
\tilde{t}_{ij} = \sum_{e\in\mathcal{E}:\{i,j\}\in e} |e|.
\end{equation*}
If an edge $(i,j)\in E$ is selected $k\ge 1$ times then we add $(i,j)$ to $E_H$ and assign the weight $k(m\mathcal{P}_{ij})^{-1}$ to it. Similar to the parameter $\alpha$ for unweighted graphs, let
\begin{equation*}\label{eq: alpha tilde}
\tilde{\alpha} = \frac{1}{n} \sum_{(i,j)\in E} \tilde{t}_{ij}^{-1}.
\end{equation*}

\begin{lemma}[Upper bound on $\tilde{\a}$]\label{lem: alpha tilde bound}
Let $\mathcal{G}=(V,\mathcal{E})$ be a hypergraph. If each node of $\mathcal{G}$ belongs to at most $d$ hyperedges  then $\tilde{\alpha}\le d/2$.
\end{lemma}

Without further assumptions on $\mathcal{G}$, the bound $\tilde{\alpha}\le d/2$ is nearly optimal. To see this, consider the following example. Let $k>0$ be an integer, $n = k^2$ and $V_1,...,V_k$ be a partition of $V=\{1,...,n\}$ such that each $V_i$ contains exactly $k$ elements $V_{i1},...,V_{ik}$. 
For each $1\le i\le k$, let $\sigma_i$ be a permutation of $\{1,2,...,k\}$ given by $\sigma_i(j) = i + j$ (mode $k$). Define the set of hyperedges of $\mathcal{G}$ as a collection of subsets of the form 
$$\left\{V_{1j},V_{2\sigma_i(j)},...,V_{k\sigma_i^{k-1}(j)}\right\}, \quad 1\le i,j\le k.$$
It is easy to see that every node of $\mathcal{G}$ is contained in exactly $d = k$ hyperedges and every pair of nodes of $\mathcal{G}$ is contained in at most one hyperedge. A simple calculation shows that $\tilde{\alpha} = (d-1)/2$.

The following theorem shows that the sparsified network obtained from a hypergraph satisfies the strong spectral property. 

\begin{theorem}[Sampling hypergraphs]\label{lem: alpha bound hypergraph}
Let $\mathcal{G}=(V,\mathcal{E})$ be a hypergraph and $G=(V,E,W)$ be the weighted network induced by $\mathcal{G}$. Let $\varepsilon\in (0,1)$ and assume that each node of $\mathcal{G}$ belongs to at most $d$ hyperedges of $\mathcal{G}$.
Form a weighted graph $H$ by sampling $4dn\log n/\varepsilon^2$ edges of $G$ as described above.
Then $H$ satisfies the strong spectral property \eqref{eq: spectral property} with probability at least $1-1/n$.
\end{theorem}

\section{Calculating the number of common neighbors}\label{sec: Tij calculation}  
In this section we discuss the problem of calculating $t_{ij}$ (either exactly or approximately), especially when the network is too large to be stored in a single computer. Once $t_{ij}$ are all computed, the sampling method can be performed easily by sampling edges according to $t_{ij}$ and aggregating over all sampled edges.   

\subsection{Exact calculation}\label{subsec: exact parallel comp}
The number of common neighbors $t_{ij}$ can be  calculated efficiently by using an MPI-based distributed memory parallel algorithm in \cite{Arifuzzaman&Khan&Marathe2019}, with very little modification.  
The algorithm first carefully partitions the graph into smaller overlapping subgraphs and stores them separately in local machines. A sequential algorithm then finds all triangles in every subgraph and counts the number of common neighbors for every connected pair of nodes in that subgraph. Since an edge of the original graph may belong to different overlapping subgraphs, the counts from all local machines are then aggregated before the final result is output. The authors of  \cite{Arifuzzaman&Khan&Marathe2019} show that their algorithm scales almost linearly in the number of local machines and can handle very large graphs with billions of edges.     

\subsection{Estimation}\label{sec:approx num comm nbrs}
Depending on the strength of the local connectivity of a graph, the computational complexity of our sampling method can be further improved by approximating $t_{ij}$ instead of calculating them exactly. In this section we describe a simple method for estimating $t_{ij}$ by sampling the neighbors of either $i$ or $j$. A similar idea has been used in minwise hashing, a popular technique for efficiently estimating the Jaccard similarity between two sets  \cite{Broder1997,Broderet.al.1997,Becchettiet.al.2008,Satuluri&Parthasarathy&Ruan2011, Shrivastava&Li2014,Shrivastava&Li2015}. Although the method described here is sequential in nature, we can easily turn it into a parallel algorithm by adapting the method of \cite{Arifuzzaman&Khan&Marathe2019} discussed in Section~\ref{subsec: exact parallel comp}.   

For each pair of connected nodes $(i,j)$, denote by $N_i$ the set of neighbors of $i$ and by $N_j$ the set of neighbors of $j$. The asymmetric Jaccard similarity between $N_i$ and $N_j$ is defined by
$$
\theta_{ij} = \frac{|N_i\cap N_j|}{\min\{|N_i|,|N_j|\}}=\frac{t_{ij}}{\min\{d_i,d_j\}}.
$$
Fix a sample size $k\ge 1$ and assume that $d_i\le d_j$. If $k\ge d_i$ then simply counting the number of neighbors of $i$ that are also neighbors of $j$ gives us exactly $t_{ij}$. If $k<d_i$,
let $Z_1,...,Z_k$ be $k$ random neighbors of $i$ drawn independently and uniformly from $N_i$. We estimate $t_{ij}$ by
\begin{equation*}\label{eq: T hat}
\hat{t}_{ij} = \frac{d_i}{k}\sum_{\ell =1}^k \onevector(Z_\ell\in N_j),
\end{equation*}
where $\onevector(Z_\ell\in N_j)$ is the indicator of the event $Z_\ell\in N_j$. It is easy to see that $\hat{t}_{ij}$ is an unbiased estimate of $t_{ij}$ because $\onevector(Z_\ell\in N_j)$, $1\le \ell\le k$, are Bernoulli random variables with success probability $\theta_{ij}$. 

In order to apply Theorem~\ref{thm: uniform edge sampling}, condition \eqref{eq: estimation accuracy} must be satisfied for all edges $(i,j)\in E_G$. For those edges such that $\theta_{ij}\ge \varepsilon$ for some constant $\varepsilon$, we will show that \eqref{eq: estimation accuracy} holds with high probability if $k$ is chosen to be of order $\log n$. For those edges with $\theta_{ij}=o(1)$, $\hat{t}_{ij}$ may not satisfy \eqref{eq: estimation accuracy} if $k=O(\log n)$, therefore we calculate $t_{ij}$ directly. We use $\frac{1}{k}\sum_{\ell =1}^k \onevector(Z_\ell\in N_j)$ to check whether $\theta_{ij}$ is sufficiently large.
The estimation procedure is summarized in the following algorithm.

\begin{algo}(Estimating number of common neighbors)\label{alg: num comm nbrs}
Choose $\varepsilon\in(0,1)$ and $k\ge 1$. For each edge $(i,j$), let $i$ be the node with $d_i\le d_j$. If $d_i\le k$, calculate $t_{ij}$ directly by counting the number of elements of $N_i\cap N_j$. If $d_i>k$, sample $k$ neighbors $Z_1,...,Z_k$ of $i$ independently and uniformly from $N_i$, calculate $\hat{\theta}_{ij} = \frac{1}{k}\sum_{\ell =1}^k \onevector(Z_\ell\in N_j)$ and proceed as follows:
\begin{itemize}
\item If $\hat{\theta}_{ij}<\varepsilon$, calculate $t_{ij}$ directly by counting the number of elements of $N_i\cap N_j$. 
\item If $\hat{\theta}_{ij}\ge\varepsilon$, estimate $t_{ij}$ by $\hat{t}_{ij} = d_i\hat{\theta}_{ij}$.
\end{itemize}
\end{algo}

The following theorem provides the spectral guarantee \eqref{eq: spectral property} for the sparsified network when Algorithm~\ref{alg: num comm nbrs} is used.

\begin{theorem}[Sampling method using $\hat{t}_{ij}$]\label{thm: approx num comm nbrs}
Let $\e\in(0,1)$, $k=100\log n/\varepsilon$ and estimate the numbers of common neighbors using Algorithm~\ref{alg: num comm nbrs}. Form a weighted graph $H$ by sampling $24{\a} n\log n/\e^2$ edges of $G$ according to Theorem~\ref{thm: uniform edge sampling}. Then with probability at least $1-1/n$, $H$ satisfies the spectral property \eqref{eq: spectral property} and the computational complexity of estimating the number of common neighbors is at most 
\begin{eqnarray}\label{eq: comp complexity}
 \sum_{(i,j)\in E_G:\theta_{ij}\le \varepsilon/2}\min\{d_i,d_j\} + 100\varepsilon^{-1}|E_G|\log n.
\end{eqnarray}
\end{theorem} 

The complexity of estimating the number of common neighbors in Theorem~\ref{thm: approx num comm nbrs} is nearly linear in the number of edges $|E_G|$ (up to the $\log n$ factor) and depends on the local structure of the network  via the first term of \eqref{eq: comp complexity}. If the local connectivity of $G$ is sufficiently strong so that $\theta_{ij}\ge \varepsilon/2$ for all edges then the first term disappears. However, for networks with very weak local connectivity, such as Erd\H{o}s-R\'{e}nyi random networks, the first term of \eqref{eq: comp complexity} may be as large as $d\cdot|E_G|$, where $d$ is the average node degree. In that case, it is not clear if the computational complexity of the (sequential) estimation algorithm can be substantially improved; we leave this problem for future study. 

%

Section~\ref{sec: estimating network stats} shows the performance of the proposed sampling method using both exact and estimated numbers of common neighbors.

\section{Numerical study}\label{eq: simulation}
\subsection{Parameter $\alpha$} \label{subsec:alpha}
According to Theorem~\ref{thm: main theorem}, $\alpha$ directly controls the accuracy of our sampling method.
In this section we show that $\a$ is relatively small for many simulated and real-world networks. 

\subsubsection{Simulated networks}
We consider geometric inhomogeneous random networks (GIRG) generated from a latent space model which has been shown to exhibit several properties of real-world networks such as the strong transitivity and the power law distribution of node degrees \cite{Bringmann&Keusch&Lengler2015}. 
To model the power law, each node $i$ is assigned a weight $w_i = \d\cdot(n/i)^{1/(\b-1)}$, where $\d>0$ and $2\le \b\le 3$ are parameters. The latent positions $x_i$ are drawn uniformly at random from an $r$-dimensional torus $\mathbb{T}^r = \R^r/\mathbb{Z}^r$ equipped with the distance 
$$
d(u,v) = \max_{1\le k\le r} \min\{|u_k-v_k|,1-|u_k-v_k|\}.
$$  
For a parameter $\gamma>1$ and $w = \sum_{i=1}^n w_i$, an edge is independently drawn between each pair of nodes $i,j$ with probability 
$$
p_{ij} = \min\left\{ \frac{1}{\|x_i-x_j\|^{\gamma r}}\Big(\frac{w_iw_j}{w}\Big)^\gamma,1 \right\}.
$$
We report in Table~\ref{table: girg statistics} the value of $\alpha$, the clustering coefficient and the average node degree (averaged over 20 replications) of networks generated from GIRG with parameter $\delta=\gamma=2$, $r=3$, $\beta=2.5$ and $n=100,500,1000,2000,4000$. Table~\ref{table: girg statistics} shows that while the network size and average node degree increase, the value of $\alpha$ increases mildly from 1.84 to 2.65 and the clustering coefficient decreases from 0.59 to 0.52.  
\begin{table}
\renewcommand{\arraystretch}{1.4}
\begin{tabular}{|l|c|c|c|c|c|}
\hline
Network size & 100 & 500 &1000 &2000 &4000\\ \hline
Parameter $\alpha$ & 1.84& 2.34&2.49 &2.59 &2.65  \\ \hline 
Clustering coefficient &0.59& 0.53& 0.52& 0.52& 0.52 \\ \hline
Average degree& 27.64& 36.42& 38.89& 40.89& 42.57 \\ \hline    
\end{tabular}
\caption{Statistics of networks generated from the GIRG with $\delta=\gamma=2$, $r=3$ and $\beta=2.5$, averaged over 20 replications.}
\label{table: girg statistics}
 \end{table}

\subsubsection{Real-world networks}  
We further report in Table~\ref{tb: stats} the value of $\a$, the clustering coefficient and the average degree of several well-known real-world networks: karate club network  \cite{Zachary1977}, dolphins network \cite{Lusseau2003}, political blogs network \cite{Adamic05}, Facebook ego network \cite{McAuley&Leskovec2012}, Astrophysics collaboration network \cite{Leskovec&Kleinberg&Faloutsos2007},  Enron email network \cite{Klimt&Yang2004}, Twitter Social circles \cite{Yang&Leskovec2012}, Google+ social circles \cite{Yang&Leskovec2012}, DBLP collaboration network  \cite{Yang&Leskovec2012} and LiveJournal social network \cite{Yang&Leskovec2012}. Again, we observe that while the network size and average node degree vary, $\a$ and the clustering coefficient are very stable, with the value of $\a$ between 1.40 and 4.38 and the value of the clustering coefficient between 0.26 and 0.63, respectively.  
 
\begin{table}[!ht]
\renewcommand{\arraystretch}{1.2}
\begin{center}
\begin{tabular}{|l|c|c|c|c|}\hline
\makecell{Data}& $n$ & Average degree& $c$ & $\alpha$ \\
\hline
Karate club & $34$ & $4.59$ & $0.57$ & $1.46$ \\
\hline
Dolphins & $62$& $5.13$&$0.26$&$1.65$\\
\hline
Political blogs & $1490$ &$22.44$ & $0.32$& $3.04$ \\
\hline
Facebook ego  &$4039$ & $43.69$&$0.61$&$1.96$ \\
\hline
Astrophysics collaboration & $18771$ & $21.10$ & $0.63$ & $1.96$ \\
\hline
Enron email &$36692$& $10.02$ & $0.50$ & $1.59$\\
\hline
Twitter &$81306$& $33.02$ & $0.57$ & $2.33$\\
\hline
Google+ &$107614$& $227.45$ & $0.49$ & $4.38$\\
\hline
DBLP collaboration &$317080$& $6.62$ & $0.63$ & $1.40$\\
\hline
LiveJournal &$3997962$& $17.35$ & $0.28$ & $3.65$\\
\hline
\end{tabular}
\end{center}
\caption{Statistics of some real-world networks.}
\label{tb: stats}
\end{table}
  

\subsection{Accuracy of network sampling} \label{sec: estimating network stats}

In this section, we compare the performance of our sampling method that uses the number of common neighbors (CN), the uniform sampling (UN) and a version of CN that uses $\hat{t}_{ij}$ defined in \eqref{eq: T hat} to approximate the number of common neighbors (CNA). We use these methods to sparsify networks, both simulated and real-world, and then measure the accuracy of the resulting sparsified networks by comparing their Laplacians with those of the original networks. Motivated naturally by the strong spectral property \eqref{eq: spectral property}, for a connected network $G$ and its sparsification $H$, we report the following relative error
\begin{eqnarray}\label{eq:rel error}
\text{Relative error} = \max_{x: L_Gx\neq 0}\frac{x^\tran(L_H-L_G)x}{x^\tran L_G x} = \|L_G^{-1/2}(L_H-L_G)L_G^{-1/2}\|,
\end{eqnarray}  
where $L_G^{-1/2}$ is the square root of the Moore–Penrose pseudo-inverse $L_G^{-1}$ of $L_G$. This error reflects the accuracy of $H$ in preserving the structure of $G$. Since calculating the relative error involves inverting the Laplacian, we consider in this section only networks of relatively small sizes.    


\subsubsection{Simulated networks}\label{sec:latant space sbm}
We first analyze the performance of CN, CNA and UN on random networks generated from the latent space stochastic block model \cite{Ng.et.al.latentspace.SBM.2018}, which has been shown to capture important characteristics of real-world networks. Specifically, we assume that nodes are partitioned into three disjoint groups or communities, and conditioning on the community labels, subsetworks corresponding to the communities follow the GIRG model defined in Section~\ref{subsec:alpha} with $\gamma=\delta=1$, $\beta =3$ and $r=10$. Edges between nodes in different communities are independently drawn with the same probability adjusted so that the ratio of the expected numbers of edges between communities and within communities equals $\rho\in\{0.01,0.1\}$, which measures the strength of the community structure. 

For each $\rho\in\{0.01,0.1\}$, we consider three settings corresponding to different community size ratios $(1/20,9/20,1/2)$, $(1/10,2/5,1/2)$ and $(1/3,1/3,1/3)$. In the first two settings, network communities are of very different sizes, while in the last setting, all communities are of the same size $n/3$.     
The networks generated in these settings are relatively dense for the sampling purpose, with expected degree ranging approximately from 300 to 500. 
We vary the sample size $m$ by setting $m=\tau n$, where $\tau$ is the sample size factor taking values from 10 to 210. To approximate the number of common neighbors for CNA, we sample $k=50$ neighbors using Algorithm~\ref{alg: num comm nbrs}.
 
Figure~\ref{fig:relative error random network} shows the relative error averaged over 10 repetitions of CN, CNA and UN in three settings and different values of $\rho$.
We observe that as the sample size $m$ increases, all methods perform better, with CN slightly better than CNA, and both methods are more accurate than UN when the communities are of different sizes and especially when $\rho=0.01$. This is because when $\rho$ is small and one community is of much smaller size than the others, UN focuses on sampling edges within large communities, mostly ignoring edges between communities and within the smallest community. In contrast, CN and CNA sample more edges within the smallest community and between communities because they have fewer common neighbors, resulting in better estimates of the Laplacian. However, when all communities are of the same size, UN tends to perform better than CN and CNA. This is perhaps because no part of any balanced and dense network needs to be sampled much more frequently than others, and UN often performs well in this case  \cite{Sadhanala&Wang&Tibshirani2016}. However, real networks are usually far from balanced, and UN may be much less accurate than CN and CNA when they are applied to these networks, as we will show next.    

\begin{figure}[H]
\centering
\begin{subfigure}{0.33\linewidth}
\centering
\includegraphics[trim= 0 20 20 50,clip,width=1\textwidth]{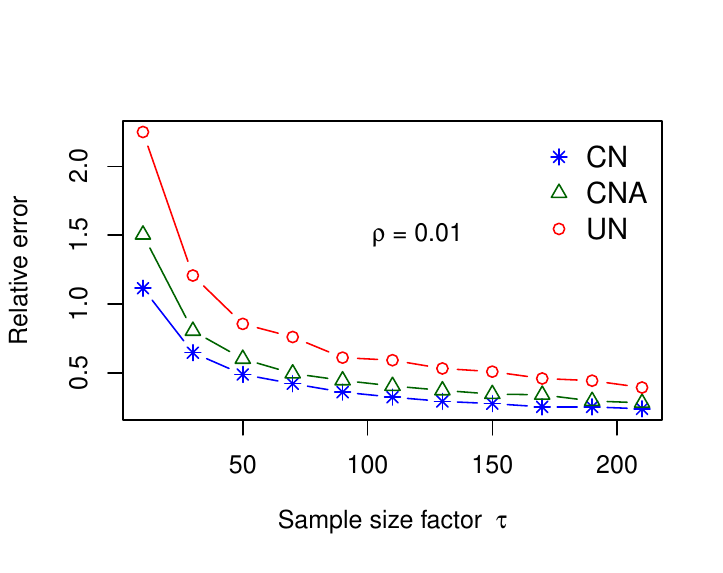}
\label{fig:ubl001}
\end{subfigure}%
\begin{subfigure}{0.33\linewidth}
\centering
\includegraphics[trim=0 20 20 50,clip,width=1\textwidth]{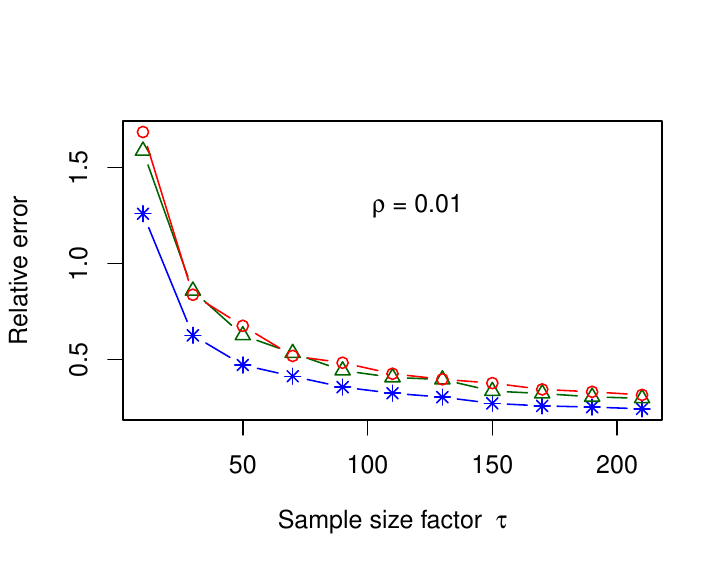}
\label{fig:ublm001}
\end{subfigure}%
\begin{subfigure}{0.33\linewidth}
\centering
\includegraphics[trim=0 20 20 50,clip,width=1\textwidth]{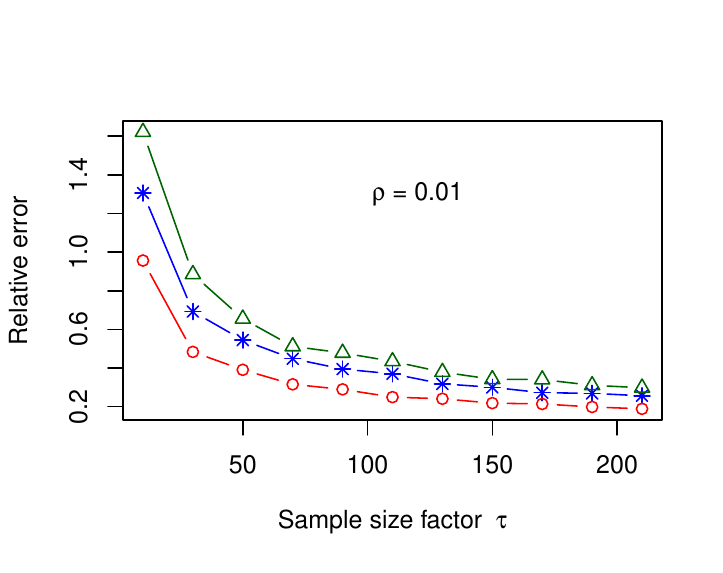}
\label{fig:bl001}
\end{subfigure} \\
\begin{subfigure}{0.33\linewidth}
\centering
\includegraphics[trim= 0 20 20 50,clip,width=1\textwidth]{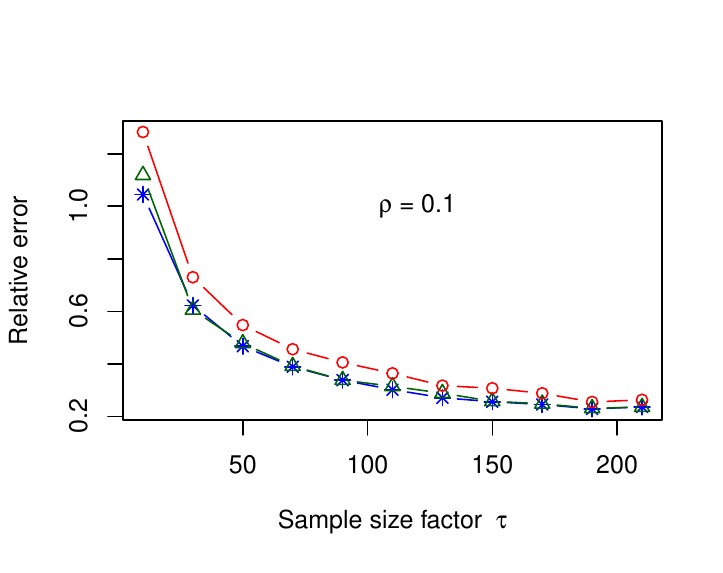}
\label{fig:ubl001}
\end{subfigure}%
\begin{subfigure}{0.33\linewidth}
\centering
\includegraphics[trim=0 20 20 50,clip,width=1\textwidth]{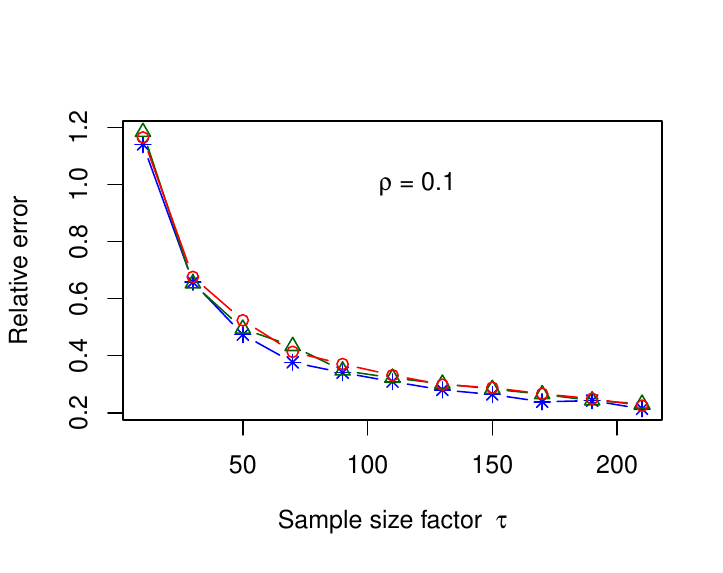}
\label{fig:ublm001}
\end{subfigure}%
\begin{subfigure}{0.33\linewidth}
\centering
\includegraphics[trim=0 20 20 50,clip,width=1\textwidth]{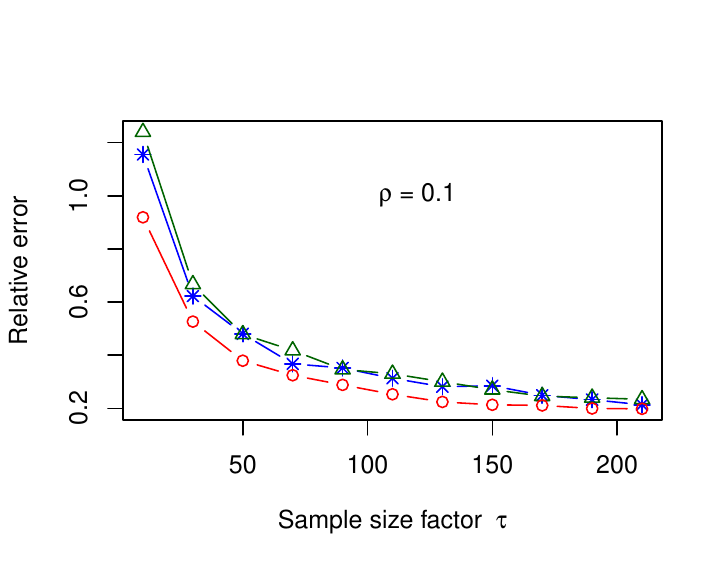}
\label{fig:bl001}
\end{subfigure}%
\caption{Accuracy of different sampling methods on random networks.}
\label{fig:relative error random network}
\end{figure}   
     
\subsubsection{Real-world networks}
We further compare the performance of CN, CNA and UN on the political blogs network \cite{Adamic05} and the Facebook ego network \cite{McAuley&Leskovec2012}, the two largest networks in Table~\ref{tb: stats} for which inverting the Laplacian can be done reasonably fast. (Note that the matrix inversion is only needed for calculating the relative error in \eqref{eq:rel error} while our proposed methods can easily handle all networks in Table~\ref{tb: stats}.) Similar to the analysis in the previous section, we vary the sample size $m$ by setting $m=\tau n$ with $\tau\in [10,50]$. To approximate the number of common neighbors for CNA, we sample $k=20$ neighbors according to  Algorithm~\ref{alg: num comm nbrs}. Figure~\ref{fig:relative error real network} shows that as $\tau$ increases,  all three methods perform better, with CN slightly better than CNA, both having much smaller errors than UN. 

\begin{figure}[H]
\centering
\begin{subfigure}{0.35\linewidth}
\centering
\includegraphics[trim=0 20 20 20,clip,width=1\textwidth]{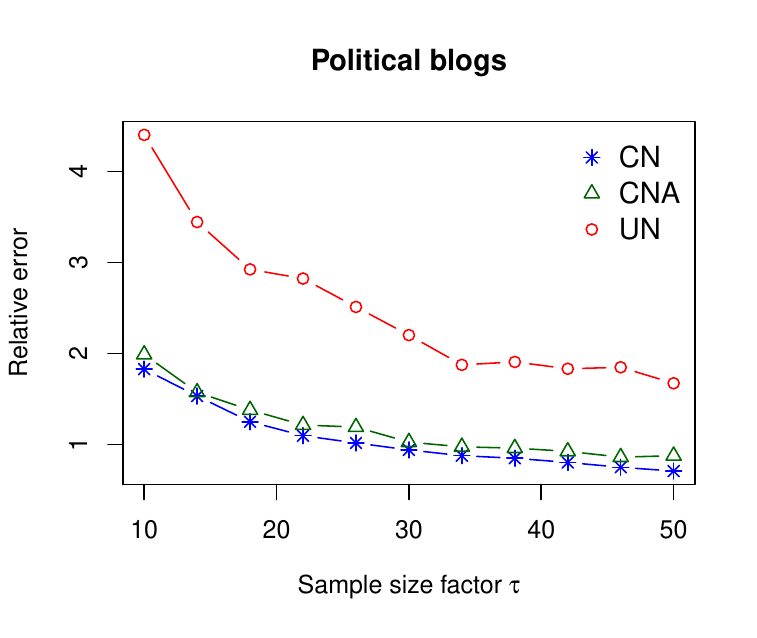}
\label{fig:ublm001}
\end{subfigure}%
\begin{subfigure}{0.35\linewidth}
\centering
\includegraphics[trim=0 20 20 20,clip,width=1\textwidth]{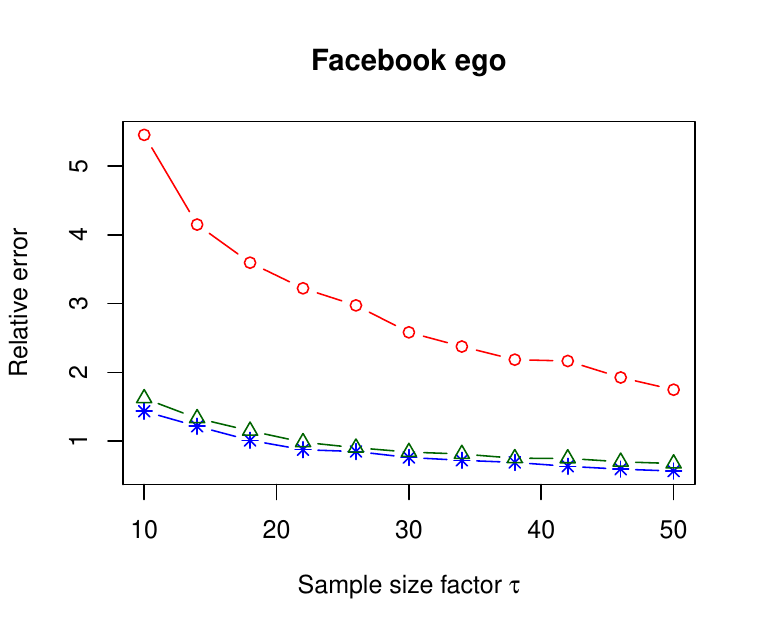}
\label{fig:bl001}
\end{subfigure} 
\caption{Accuracy of different sampling methods on real networks.}
\label{fig:relative error real network}
\end{figure}   

Note that the performance gap between UN and the proposed methods is much more visible on the real networks than on simulated networks shown in the previous section. This is probably due to the significant difference between the distributions of the number of common neighbors of simulated and real networks. As shown in Figure~\ref{fig:ncn},  most of the edges of the simulated networks have very large numbers of common neighbors, and for such nodes, the uniform sampling performs well (this is partially explained in the discussion following Theorem~\ref{thm: uniform edge sampling}). In contrast, edges of the real networks considered here have relatively smaller numbers of  common neighbors, resulting in the much worse performance of the uniform sampling. The numerical results on real networks show that sampling using the number of common neighbors may be very useful in practice.

\begin{figure}[H]
\centering
\begin{subfigure}{0.33\linewidth}
\centering
\includegraphics[trim=0 10 20 20,clip,width=1\textwidth]{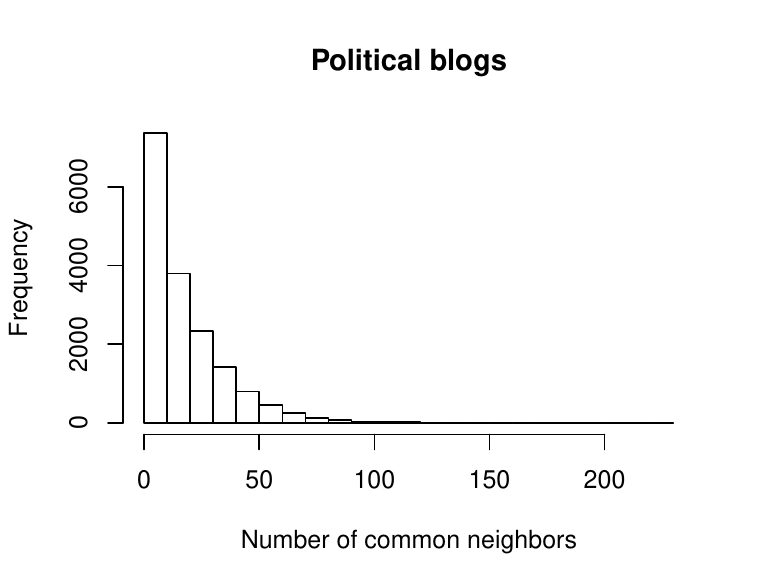}
\label{fig:ublm001}
\end{subfigure}%
\begin{subfigure}{0.33\linewidth}
\centering
\includegraphics[trim=0 10 20 20,clip,width=1\textwidth]{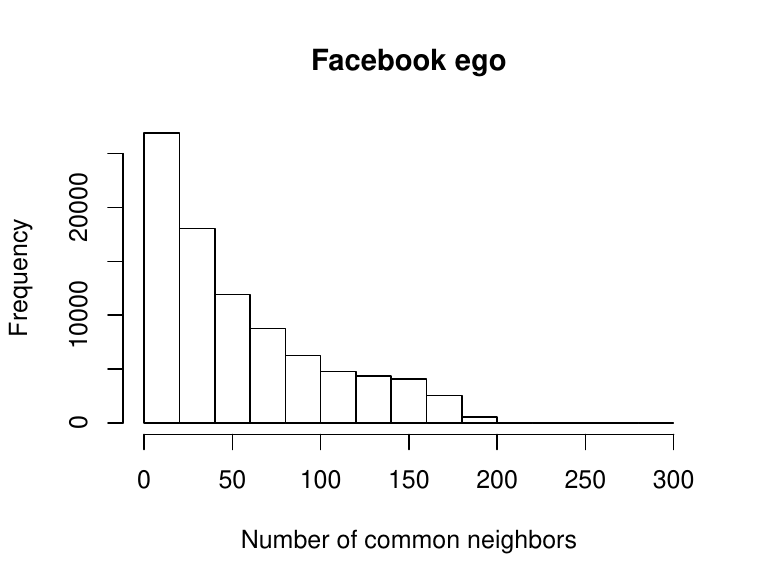}
\label{fig:bl001}
\end{subfigure} 
\begin{subfigure}{0.33\linewidth}
\centering
\includegraphics[trim=0 10 20 20,clip,width=1\textwidth]{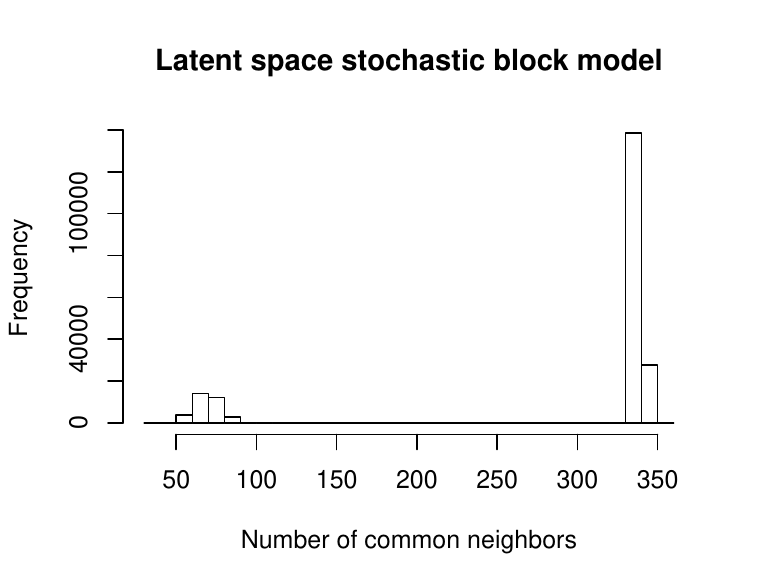}
\label{fig:bl001}
\end{subfigure} 
\caption{Distributions of the number of common neighbors for the political blogs network, the Facebook ego network  and a network generated from a latent space stochastic block model described in Section~\ref{sec:latant space sbm} with community size ratios $(1/3,1/3,1/3)$ and $\rho = 0.1$.}
\label{fig:ncn}
\end{figure}

\section{Discussion}
In this paper we study an edge sampling algorithm that uses only the number of common neighbors. This simple statistic provides an easy way to measure the strength of network local connectivity through parameter $\alpha$, which directly controls the accuracy of the sampling method. However, in practice we often have access to not only the numbers of common neighbors but also neighborhood networks around edges. In that case, we should use the information from these local networks, provided that it is available or easily computed, because it contains more structural information of the network than just the numbers of common neighbors. Measuring the strength of local connectivity through local networks is more challenging and we leave it for future work.    


\appendix


\section{Proofs of results in Section~\ref{sec: theory}}\label{subsec: proof of main theorems}

Theorem~\ref{thm: main theorem} directly follows from Theorem~\ref{thm: uniform edge sampling} with $\hat{t}_{ij} = t_{ij}$ for all edges $(i,j)\in E_G$ and $C=1$. To prove Theorem~\ref{thm: uniform edge sampling}, we use the following result about the concentration of the sum of random matrices \cite{VershyninNote2009}.

\begin{theorem}[Concentration of sum of matrices]\label{thm: concentration of sum of matrices}
Let $Y_1,...,Y_m$ be independent $n\times n$ random positive semidefinite matrices such that $\|Y_k\|\le M$ for all $1\le k \le m$. Let $S_m=\sum_{k=1}^m Y_k$ and $E = \sum_{k=1}^m \| \E Y_k\|$. Then for every $\varepsilon\in (0,1)$ we have
$$
\Pr{\|S_m-\E S_m\| > \varepsilon E} \le n\cdot \exp\left(\frac{-\varepsilon^2 E}{4M}\right).
$$
\end{theorem}

\begin{proof}[Proof of Theorem~\ref{thm: uniform edge sampling}]
Let $X$ be a random matrix such that
$$X = \frac{1}{p_{ij}}(e_i-e_j)(e_i-e_j)^\tran \quad \text{with probability } \ \hat{p}_{ij},$$
where $(i,j)\in E_G$, $\{e_i, 1\le i\le n\}$ are standard basis vectors (the $i$th entry of $e_i$ is one and all other entries are zero), and 
\begin{equation}\label{eq: p hat}
\hat{p}_{ij} = \frac{\frac{2}{\hat{t}_{ij}+2}}{\sum_{(i,j)\in E_G}\frac{2}{\hat{t}_{ij}+2}}.
\end{equation}
Then
\begin{equation}\label{eq: E X}
\E X = \sum_{(i,j)\in E_G} \hat{p}_{ij} \times \frac{1}{\hat{p}_{ij}}(e_i-e_j)(e_i-e_j)^\tran = L_G.
\end{equation}
Let $X_k$ be $m$ independent copies of $X$. By the sampling scheme we have 
$$L_H = \frac{1}{m}\sum_{k=1}^m X_k, \quad \E L_H =  L_G.$$
Denote by $L_G^{-1}$ the Moore-Penrose pseudoinverse of $L_G$ and by $L_G^{-1/2}$ the squared root of $L_G^{-1}$.
Note that the kernel of the map $L_G$ is a one-dimensional vector space spanned by the all-one vector $\onevector$ and it is contained in the kernel of $L_H$. Therefore the strong spectral property \eqref{eq: spectral property} is equivalent to
\begin{equation}\label{eq: spectral property 1}
(1-\e)I_\onevector \preceq  \frac{1}{m}\sum_{k=1}^m  L_G^{-1/2} X_k L_G^{-1/2} \preceq (1+\e) I_\onevector,
\end{equation}
where $I_\onevector = I-(1/n)\onevector\onevector^\tran$ is the identity map on the $(n-1)$-dimensional subspace orthogonal to the all-one vector $\onevector$. Here, we write $U\preceq V$ if $V-U$ is positive semidefinite. 

To prove \eqref{eq: spectral property 1}, we apply Theorem~\ref{thm: concentration of sum of matrices} to  $Y_k := L_G^{-1/2} X_k L_G^{-1/2}$. Since $X_k\succeq 0$ and $\E X_k= L_G$ by \eqref{eq: E X}, it follows that $Y_k \succeq 0$ and $\|\E Y_k\|= \|I_\onevector\| = 1$. To bound $\|Y_k\|$, note that $Y_k$ takes one of the following matrix values
$$
\frac{1}{\hat{p}_{ij}} \left(L_G^{-1/2} (e_i -  e_j)\right)
\left(L_G^{-1/2} (e_i - e_j)\right)^\tran, \quad (i,j)\in E_G.
$$ 
By \eqref{eq: p hat} and \eqref{eq: alpha hat} we have $1/\hat{p}_{ij}= n\hat{\a}(\hat{t}_{ij}+2)/2$. Therefore
\begin{equation}\label{eq: bound Yk 1}
\|Y_k\| \le \max_{(i,j)\in E_G} \frac{n\hat{\a}(\hat{t}_{ij}+2)}{2} \cdot(e_i-e_j)^\tran L_G^{-1}(e_i-e_j).
\end{equation}
Note that $(e_i-e_j)^\tran L_G^{-1}(e_i-e_j)$ is the effective resistance of the edge between $i$ and $j$ \cite{Ghosh&Boyd&Saberi2008}. We claim that it is upper bounded by $2/(t_{ij}+2)$. To show that, let $N_{ij}$ be the set of common neighbors of $i$ and $j$. Denote by $G_{ij}=(V_{ij},E_{ij})$ the subgraph of $G$ such that
$$ V_{ij}=\{i,j\}\cup N_{ij}, \qquad E_{ij} = \{(i,j),(i,k),(j,k): k\in N_{ij}\}.$$ 
Thus, $G_{ij}$ consists of an edge and $t_{ij}$ paths of length two between $i$ and $j$. It is easy to see that the effective resistance $(e_i-e_j)^\tran L^{-1}_{G_{ij}}(e_i-e_j)$ of the edge between $i$ and $j$ in $G_{ij}$ is $2/(t_{ij}+2)$. Indeed, let $x = L^{-1}_{G_{ij}}(e_i-e_j)$. Then $L_{G_{ij}} x = e_i-e_j$ and by comparing the $i$-th and $j$-th components of $L_{G_{ij}} x$ and $e_i-e_j$, we have  
$$
(t_{ij}+1)x_i - x_j - \sum_{k\in N_{ij}} x_k = 1, \qquad x_i- (t_{ij}+1)x_j+\sum_{k\in N_{ij}} x_k = 1.
$$
Adding these equalities, we obtain that the effective resistance of the edge between $i$ and $j$ in $G_{ij}$ is $(e_i-e_j)^\tran x=x_i-x_j = 2/(t_{ij}+2)$. Since $G_{ij}$ is a subgraph of $G$ and adding edges does not increase the effective resistance (see e.g. Corollary 9.13 in \cite{Levin&Peres.mixing.2017}), it follows that
$$
(e_i-e_j)^\tran L_G^{-1}(e_i-e_j) \le (e_i-e_j)^\tran L^{-1}_{G_{ij}}(e_i-e_j) = \frac{2}{t_{ij}+2}.
$$
Together with \eqref{eq: bound Yk 1} this implies $\|Y_k\| \le n\hat{\a}(\hat{t}_{ij}+2)/(t_{ij}+2)\le n\hat{\a} C$. Therefore by Theorem~\ref{thm: concentration of sum of matrices} we have
$$
\Pr{\left\|\frac{1}{m}\sum_{k=1}^m Y_k-I_{\onevector}\right\| >\varepsilon } \le n\cdot \exp\left(\frac{-\varepsilon^2 m}{4C\hat{\a} n}\right).
$$
Inequality \eqref{eq: spectral property 1} then follows by choosing $m=8C\hat{\a} n \log n /\varepsilon^2$.
\end{proof}


\section{Proofs of results in Section~\ref{sec: bound on alpha}}

\begin{proof}[Proof of Theorem~\ref{lem: upper bound er}] 
Let $(i,j)$ be an edge of $G$ such that 
$\kappa(i,j)\ge \kappa_0$ or equivalently $W_1(m_i,m_j) \le 1-\kappa_0$. Recall that 
$$
W_1(m_i,m_j) = \inf_{\xi\in\Pi (m_i,m_j)}\sum_{(k,k')\in V\times V} d(k,k')\xi(k,k'),
$$  
where $\Pi(m_i,m_j)$ is the set of all probability measures on $V\times V$ with marginals $m_i$ and $m_j$. For every $\xi\in\Pi(m_i,m_j)$,
$$
\xi(k,k) \le \min\left\{\sum_{k'\in V} \xi(k',k),\sum_{k'\in V} \xi(k,k')\right\} = \min\{m_j(k),m_i(k)\}.
$$
Since $m_i$ and $m_j$ are uniform measures with supports being the sets of neighbors of $i$ and $j$, respectively, if $k$ is not a common neighbor of $i$ and $j$ then 
$\min\{m_j(k),m_i(k)\} = 0$; when $k$ is one of the $t_{ij}$ common neighbors of $i$ and $j$ then $\min\{m_j(k),m_i(k)\} = \min\{1/d_i,1/d_j\}$. Therefore
$$
\sum_{k\in V}\xi(k,k) \le t_{ij}\cdot \min\{1/d_i,1/d_j\},
$$
which implies
\begin{eqnarray*}
\sum_{(k,k')\in V\times V} d(k,k')\xi(k,k') &=& \sum_{k\neq k'} d(k,k')\xi(k,k') \\
&\ge& \sum_{k\neq k'} \xi(k,k')\\
&=& 1 - \sum_{k\in V}\xi(k,k) \\
&\ge& 1 - t_{ij}\cdot \min\{1/d_i,1/d_j\}.
\end{eqnarray*}
Taking the infimum over all $\xi\in\Pi(m_i,m_j)$, we have
$$
1-\kappa_0 \ge W_1(m_i,m_j) \ge 1-t_{ij}\cdot \min\{1/d_i,1/d_j\},
$$
or $t_{ij}\ge \kappa_0/\min\{1/d_i,1/d_j\}$.
Denote by $\mathcal{E}$ the set of all edges of $G$ such that $\kappa(i,j)\ge \kappa_0$ and by $\mathcal{E}^c$ its complement. Since $|\mathcal{E}^c|\le Cn$ by assumption, 
\begin{eqnarray*}
\alpha &=& \frac{1}{n}\sum_{(i,j)\in\mathcal{E}} \frac{2}{2+t_{ij}} + \frac{1}{n}\sum_{(i,j)\in\mathcal{E}^c} \frac{2}{2+t_{ij}} \\
&\le& \frac{2}{\kappa_0n}\sum_{(i,j)\in \mathcal{E}} \min\{1/d_i,1/d_j\} + C\\
&\le& \frac{1}{\kappa_0n}\sum_{i\in V}\sum_{j\in N_i} \min\{1/d_i,1/d_j\} + C \\
&\le& \frac{1}{\kappa_0} + C.
\end{eqnarray*}
For the last inequality, we use the fact that $\sum_{j\in N_i} \min\{1/d_i,1/d_j\} \le 1$. 
\end{proof}

For proving Theorem~\ref{lem: lower bound resistance sum}, we need the following lemma.

\begin{lemma}\label{lem: simple inequality}
For positive numbers $x_1,x_2,...,x_k$ the following inequality holds
\begin{equation*}\label{eq: simple inequality}
\left(x_1+x_2+\cdots+x_k\right)\left(\frac{1}{x_1}+\frac{1}{x_2}+\cdots+\frac{1}{x_k}\right)\ge k^2.
\end{equation*}
The two sides are equal if and only if $x_1=x_2=\cdots = x_n$.
\end{lemma}
\begin{proof}[Proof of Lemma~\ref{lem: simple inequality}]
Using the inequality of arithmetic and geometric means, we have
$$
x_1+x_2+\cdots+x_k\ge k(x_1x_2\cdots x_k)^{1/k}, \quad
\frac{1}{x_1}+\frac{1}{x_2}+\cdots+\frac{1}{x_k} \ge k(x_1x_2\cdots x_k)^{-1/k}.
$$
Lemma~\ref{lem: simple inequality} follows directly from these inequalities.
\end{proof}

\begin{proof}[Proof of Theorem~\ref{lem: lower bound resistance sum}]
For each node $i$, denote by $N_i$ and $t_i$ the set of neighbors of $i$ and the number of triangles that contain $i$, respectively. Using Lemma~\ref{lem: simple inequality}, we have
$$
\sum_{j\in N_i}\frac{2}{t_{ij}+2} 
\ge \frac{2|N_i|^2}{\sum_{j\in N_i} (t_{ij}+2)} 
= \frac{d_i^2}{t_i + d_i} 
\ge \frac{1}{2c_i+\frac{1}{d_i}}.
$$  
Summing over all nodes $i$ and applying Lemma~\ref{lem: simple inequality} again, we obtain
$$
\sum_{(i,j)\in E}\frac{4}{t_{ij}+2} 
\ge \sum_{i\in V} \frac{1}{2c_i+\frac{1}{d_i}}
\ge \frac{|V|^2}{\sum_{i\in V}\left(2c_i+\frac{1}{d_i}\right)}
= \frac{n}{2c+ \frac{1}{n}\sum_{i\in V}\frac{1}{d_i}}.
$$
The proof is complete by dividing both sides of this inequality by $2n$.
\end{proof}

\section{Proofs of results in Section~\ref{sec: random graphs}}

\begin{proof}[Proof of Theorem~\ref{thm: alpha bound random graph}]
We rewrite $S := n\alpha$ as follows:
\begin{eqnarray*}
S = \sum_{(i,j)\in E_G} \frac{2}{t_{ij}+2} =  \sum_{i<j} \frac{2A_{ij}}{t_{ij}+2} =: \sum_{i<j} Y_{ij}.
\end{eqnarray*}
The proof consists of two parts: showing that $S\le 2\E S$ with high probability and upper bounding $\E S$.
For the first part, note that $S$ is a sum of $n(n-1)/2$ weakly depdendent random variables $Y_{ij} = 2A_{ij}/(t_{ij}+2)$, where $Y_{ij}$ and $Y_{i'j'}$ are independent if $i'\neq i$ and $j'\neq j$. To deal with the dependence among $Y_{ij}$, we will use the moment method (see for example \cite{Warnke2017}).  

For notational simplicity, we denote $\gamma=\{i,j\}$ as a set of two elements $i,j$ and write $S = \sum_\gamma Y_\gamma$. With the new notation, $Y_\gamma$ and $Y_{\gamma'}$ are independent if $\gamma\cap\gamma' = \emptyset$. 
For a positive integer $k$, let
$$
M_k = \sum_{(\gamma_1,...,\gamma_k)} \prod_{i=1}^k Y_{\gamma_i},
$$
where the sum is over all $k$-tuples $(\gamma_1,...,\gamma_k)$ such that $\gamma_i\cap\gamma_j = \emptyset$ if $i\neq j$. Since  $Y_{\gamma_1},Y_{\gamma_2},...,Y_{\gamma_k}$ are independent by construction,
\begin{equation}\label{eq: expected m bound}
\E M_k =\sum_{(\gamma_1,...,\gamma_k)} \prod_{i=1}^k \E Y_{\gamma_i} \le \left(\sum_\gamma \E Y_\gamma\right)^k = \left(\E S\right)^k.
\end{equation}
Denote by $\mathcal{E}$ the event that $S >2\E S$. When $\mathcal{E}$ occurs,
\begin{eqnarray*}
M_{k+1} &=& \sum_{(\gamma_1,...,\gamma_k)} \prod_{i=1}^k Y_{\gamma_i}\left(\sum_{\gamma}Y_{\gamma}-\sum_{\gamma\cap \gamma_i \neq \emptyset \text{ for some }1\le i\le k}Y_{\gamma}\right) \\
&\ge& \sum_{(\gamma_1,...,\gamma_k)} \prod_{i=1}^k Y_{\gamma_i}\left(2\E S -\sum_{\gamma\cap \gamma_i \neq \emptyset \text{ for some }1\le i\le k}Y_{\gamma}\right).
\end{eqnarray*}
For each $(\gamma_1,...,\gamma_k)$ we have
$$
\sum_{\gamma\cap \gamma_i \neq \emptyset \text{ for some }1\le i\le k}Y_{\gamma} \le 
\sum_{i=1}^k \sum_{\gamma\cap \gamma_i\neq \emptyset} Y_\gamma \le k \cdot \max_{\gamma'} \sum_{\gamma\cap \gamma'\neq \emptyset} Y_\gamma.
$$
Let $Z = \max_{i} \sum_{j=1}^n A_{ij}$ be the maximal node degree. Since $Y_\gamma \le A_\gamma$, 
$$
\max_{\gamma'} \sum_{\gamma\cap \gamma'\neq \emptyset} Y_\gamma \le \max_{\gamma'} \sum_{\gamma\cap \gamma'\neq \emptyset} A_\gamma \le 2Z.
$$
Therefore if $\mathcal{E}$ occurs then
\begin{equation}\label{eq: Mk iter bound}
M_{k+1} \ge \sum_{(\gamma_1,...,\gamma_k)} \prod_{i=1}^k Y_{\gamma_i}\left(2\E S -2kZ\right) = M_k\left(2\E S-2kZ\right).
\end{equation}

We now show that $2\E S-2kZ \ge (3/2)\cdot\E S>0$ with high probability for $k=O(\log n)$, so the above inequality can be applied repeatedly to obtain a desired lower bound for $M_k$. 
Since upper diagonal elements of $A$ are independent, by Jensen's inequality we have
\begin{eqnarray*}
\E S
= \sum_{i<j} P_{ij}\cdot \E\left[ \frac{2}{t_{ij}+2}\right] 
\ge \sum_{i<j}  \frac{2P_{ij}}{\E t_{ij}+2}\ge \ \frac{2}{\max_{i,j}\E t_{ij}+2}\cdot \sum_{i<j} P_{ij}.
\end{eqnarray*}
The second inequality of \eqref{eq: max condition} then implies
$$
\E S \ge C \Delta\log n.
$$
Let $\mathcal{E}_1$ be the event that $Z\le 2 \Delta$. Since $\Delta>C\log n$, it follows from the Chernoff and union bounds that
\begin{equation}\label{eq: max degree bound}
\P(\mathcal{E}_1) = \P(Z\le 2\Delta) \ge 1-\frac{1}{2n}.
\end{equation}
When $\mathcal{E}_1$ occurs, 
$$
2\E S - 2kZ = \frac{3\E S}{2} + \frac{\E S}{2} - 2kZ \ge \frac{3\E S}{2} + \frac{ C \Delta\log n}{2} - 4k\Delta \ge \frac{3\E S}{2}
$$
for all $k\le m$, where $m=\lfloor (C/8)\log n\rfloor$ is the largest integer not greater than $(C/8)\log n$. Therefore if $\mathcal{E}\cap\mathcal{E}_1$ occurs then $M_1=S>2\E S$ and by applying \eqref{eq: Mk iter bound} repeatedly,
$$
M_{m} \ge M_{m-1}\cdot\frac{3\E S}{2} \ge \cdots\ge \left(\frac{3\E S}{2}\right)^{m}.
$$
Using Markov's inequality, \eqref{eq: expected m bound} and \eqref{eq: max degree bound}, we have
\begin{eqnarray*}
\P(S>2\E S) &\le& \P(S>2\E S,Z\le 2\Delta) + \P(Z>2\Delta) \\
&\le& \P\left(M_{m} \ge \left[\frac{3\E S}{2}\right]^{m}\right) + \frac{1}{2n}\\
&\le& \left(\frac{2}{3}\right)^m + \frac{1}{2n} \\
&\le& \frac{1}{n},
\end{eqnarray*}
where the last inequality holds for sufficiently large $C$. Thus, $S\le 2\E S$ with probability at least $1-1/n$.

It remains to bound 
\begin{eqnarray*}
\E S = \sum_{i<j} 2P_{ij}\cdot \E\left[ \frac{1}{t_{ij}+2}\right].
\end{eqnarray*}
For every pair of nodes $(i,j)$ we have 
\begin{eqnarray}\label{eq: deviation 1/T}
\left|\E \frac{1}{t_{ij}+2}-\frac{1}{\E t_{ij}+2}\right| = \frac{1}{\E t_{ij}+2}\cdot\left|\E \frac{t_{ij}-\E t_{ij}}{t_{ij}+2}\right| \le \frac{1}{\E t_{ij}+2}\cdot\E \frac{|t_{ij}-\E t_{ij}|}{t_{ij}+2}. 
\end{eqnarray}
Since $t_{ij}$ is the sum of $n$ independent Bernoulli random variables, by Chernoff bound,
$$
\P\left(t_{ij}\le\E t_{ij}/2\right) \le \exp\left(-\frac{\E t_{ij}}{8}\right). 
$$
Consider the function $g(x) = |x-\E t_{ij}|/(x+2)$ with $x\ge 0$. It is easy to show that $g(x)\le \E t_{ij}$ if $x\le\E t_{ij}/2$ and
$g(x) \le 1$ if $x> \E t_{ij}/2$.
Therefore
\begin{eqnarray*}
\E \frac{|t_{ij}-\E t_{ij}|}{t_{ij}+2} &=& \E g(t_{ij})\cdot I(t_{ij}\le\E t_{ij}/2) + \E g(t_{ij})\cdot I(t_{ij}>\E t_{ij}/2)  \\
&\le& \E t_{ij}\cdot \P(t_{ij}\le\E t_{ij}/2) + \P(t_{ij}>\E t_{ij}/2) \\
&\le& \E t_{ij}\cdot\exp\left(-\frac{\E t_{ij}}{8}\right)+1 \\
&\le& 4.
\end{eqnarray*}
The last inequality follows from the fact that $x\cdot\exp(-x/8)\le 3$ for all $x\ge 0$. By \eqref{eq: deviation 1/T} and the last inequality, we get
$$
\E \frac{1}{t_{ij}+2} \le \frac{5}{\E t_{ij}+2}.
$$
Finally,
\begin{eqnarray*}
\E S = \sum_{i<j} 2P_{ij}\cdot \E\left[ \frac{1}{t_{ij}+2}\right] \le 
\sum_{i<j} \frac{10P_{ij}}{\E t_{ij}+2}
\end{eqnarray*}
and the proof is complete.
\end{proof}

\begin{proof}[Proof of Corollary~\ref{cor: geometric graph}]
We first verify the conditions in Theorem~\ref{thm: alpha bound random graph}. By \eqref{eq: geometric graph condition}, 
$$
\Delta = \max_{i} \sum_{j=1}^n P_{ij} = \max_i n_i\delta \ge C^{-1}nr_n^d\delta \ge C\log n. 
$$
Therefore the first condition of \eqref{eq: max condition} is satisfied if $C$ is sufficiently large. 
Also, since
$$
\Delta\cdot \max_{i,j}\left[1+\left(P^2\right)_{ij}\right] = \max_i n_i\delta\cdot\max_{ij}(1+n_{ij}\delta^2) \le 2Cn^2r_n^{2d}\delta^3
$$
and 
$$
\frac{1}{\log n} \sum_{i<j} P_{ij}=\frac{1}{2\log n}\sum_{i=1}^n n_{ij}\delta\ge \frac{n^2 r_n^d \delta}{2C\log n}, 
$$
the second condition of \eqref{eq: max condition} holds because
$
r_n^d\le (4C^3\delta^2\log n)^{-1}
$. Therefore  by Theorem~\ref{thm: alpha bound random graph},  with probability at least $1-1/n$, the upper bound of $\alpha$ is
\begin{eqnarray*}
\frac{1}{n}\sum_{i<j}  \frac{10P_{ij}}{\E t_{ij}+2} &=& \frac{1}{n}\sum_{i=1}^n
\sum_{j:\|x_i-x_j\|\le r_n}  \frac{5\delta}{n_{ij}\delta^2+2} \\
 &\le& \frac{1}{n} \sum_{i=1}^n
\frac{5n_i\delta}{n_{ij}\delta^2+2} \\
&\le& 5C^2\delta^{-1},
\end{eqnarray*}
and the proof is complete.
\end{proof}

\section{Proof of results in Section~\ref{sec: hypergraphs}}

\begin{proof}[Proof of Lemma~\ref{lem: alpha tilde bound}]
By the definition of $E$ and $\mathcal{E}$, we have
$$
\sum_{(i,j)\in E} \tilde{t}_{ij}^{-1} \le \sum_{e\in \mathcal{E}}\sum_{\{i,j\}\subseteq e} \tilde{t}_{ij}^{-1}.
$$
Since $\tilde{t}_{ij} \ge |e|$ for each $e\in E_\mathcal{G}$ that contains $\{i,j\}$ and there are $|e|(|e|-1)/2$ pairs $\{i,j\}\in e$, it follows from above inequality that
$$
\sum_{(i,j)\in E} \tilde{t}_{ij}^{-1} \le
\sum_{e\in \mathcal{E}}\frac{|e|-1}{2}
\le \frac{1}{2} \sum_{e\in \mathcal{E}} |e| 
\le \frac{dn}{2}.
$$ 
For the last inequality we use the assumption that each node belongs to at most $d$ hyperedges.
\end{proof}

\begin{proof}[Proof of Theorem~\ref{lem: alpha bound hypergraph}]
The proof of this theorem is similar to the proof of Theorem~\ref{thm: main theorem} with one exception that we replace $\tilde{\alpha}$ with the upper bound $d/2$ shown in Lemma~\ref{lem: alpha tilde bound}. \end{proof}

\section{Proof of results in Section~\ref{sec: Tij calculation}  }

\begin{proof}[Proof of Theorem~\ref{thm: approx num comm nbrs}] Let $S$ be the sum of 
$k$ independent Bernoulli random variables with success probability $\theta$. Then by Bernstein's inequality, for any $\delta\in [0,1]$,  
\begin{equation}\label{eq: single theta hat}
\P\left(|S-k\theta|>\delta k\theta \right) 
\le 2\exp\left(\frac{-(\delta k \theta)^2/2}{k\theta+(\delta k \theta)/3}\right)
\le 2\exp\left(\frac{-\delta^2 k \theta}{4}\right).
\end{equation}
Let $\mathcal{E}$ be the set of all edges with $\theta_{ij}<\varepsilon/2$. 
For each $(i,j)\in \mathcal{E}$, $k\hat{\theta}_{ij}$ is stochastically bounded by $S$ with $\theta = \varepsilon/2$. Therefore by \eqref{eq: single theta hat} with $\delta=\varepsilon$, 
\begin{eqnarray*}
\P(\hat{\theta}_{ij}\ge \varepsilon) \le \P(S\ge k\varepsilon) \le \P(|S-k\varepsilon/2|\ge k\varepsilon/2) \le 2\exp\left(\frac{-k\e }{8}\right).
\end{eqnarray*}
Since $k = 100\log n/\varepsilon$, by the union bound,
\begin{equation}\label{eq: thetahat bound}
\P\left(\hat{\theta}_{ij}< \varepsilon \ \text{for all} \ (i,j)\in\mathcal{E}\right) \ge 1 - 2n^2\exp\left(\frac{-k\e}{8}\right)\ge 1 - \frac{1}{2n}.
\end{equation}
Consider now $\mathcal{E}^c$, the set of nodes with $\theta_{ij}\ge \varepsilon/2$. Then using \eqref{eq: single theta hat} with $\theta = \theta_{ij}$ and $\delta=1/2$, we get
\begin{eqnarray*}
\P(|\hat{\theta}_{ij}-\theta_{ij}|\ge \theta_{ij}/2) \le \P(|S-k\theta_{ij}|\ge k\theta_{ij}/2) \le 2\exp\left(\frac{-k\theta_{ij}}{16}\right) \le 2\exp\left(\frac{-k\varepsilon}{32}\right).
\end{eqnarray*}
Therefore with $k = 100\log n/\varepsilon^3$, we have
\begin{eqnarray}
\nonumber\P\left(\Big|\frac{\hat{t}_{ij}}{t_{ij}}-1\Big|\le \frac{1}{2} \ \text{for all} \ (i,j)\in \mathcal{E}^c\right) 
&=& \P\left(\Big|\frac{\hat{\theta}_{ij}}{\theta_{ij}}-1\Big|\le \frac{1}{2} \ \text{for all} \ (i,j)\in \mathcal{E}^c\right)\\
\nonumber&\le& 2n^2 \exp\left(\frac{-k\varepsilon}{32}\right)\\
\label{eq: That bound}&\le& 1 - \frac{1}{2n}.
\end{eqnarray}
Recall that for edges $(i,j)$ with $\min\{d_i,d_j\}>k$, we calculate $t_{ij}$ directly if $\hat{\theta}_{ij}<\varepsilon$ and estimate $t_{ij}$ by $\hat{t_{ij}} = \hat{\theta}_{ij}\cdot\min\{d_i,d_j\}$ otherwise. From \eqref{eq: thetahat bound} and \eqref{eq: That bound}, we obtain that $|\hat{t}_{ij}-t_{ij}|\le 1/2$ for all $(i,j)$ with probability at least $1-1/n$. The spectral property \eqref{eq: spectral property} then follows from Theorem~\ref{thm: uniform edge sampling}.

The computational complexity of estimating all $t_{ij}$ is bounded by
\begin{eqnarray*}
\sum_{(i,j)\in \mathcal{E}} \min\{d_i,d_j\} + \sum_{(i,j)\in \mathcal{E}^c }k \le \sum_{(i,j):\theta_{ij}\le \varepsilon/2}\min\{d_i,d_j\} + 100|E_G|\log n/\varepsilon.
\end{eqnarray*}
The proof is complete.
\end{proof}

\vskip 0.2in
\bibliography{allref}
\bibliographystyle{abbrv}
\end{document}